\documentclass[notheoremnums]{dpreprint}

\newtheorem{theorem}{Theorem}[section]
\newtheorem*{theorem*}{Theorem}
\newtheorem*{proposition*}{Proposition}
\newtheorem{corollary}[theorem]{Corollary}
\newtheorem{proposition}[theorem]{Proposition}
\newtheorem{lemma}[theorem]{Lemma}

\theoremstyle{definition}
\newtheorem{definition}[theorem]{Definition}
\newtheorem{conjecture}[theorem]{Conjecture}

\newtheorem*{definition*}{Definition}
\newtheorem{remark}[theorem]{Remark}
\newtheorem*{remark*}{Remark}
\newtheorem{question}[theorem]{Question}
\newtheorem{maintheorem}{Theorem}

\usepackage{xcolor}
\usepackage{mathtools}

\newcommand{\dblank}{\mathrel{\vcenter{\hbox{\rule{1cm}{0.5pt}}}}}
\newcommand{\1}{\widehat{1}}
\newcommand{\ls}[2][1.5]{there exists $q_{n} \geq h_{n}$ such that $p(q_{n}) \geq #1q_{n} + f_{n} - #2$}
\newcommand{\len}{\ell en}

\begin{document}

\dtitle[Word complexity of weakly mixing rank-one subshifts]{Word Complexity of (Measure-Theoretically) Weakly Mixing Rank-One Subshifts}
\dauthor[Darren~Creutz]{Darren Creutz}{creutz@usna.edu}{US Naval Academy}{}
\datewritten{\today}


\dabstract{%
We exhibit, for arbitrary $\epsilon > 0$, subshifts admitting weakly mixing (probability) measures with word complexity $p$ satisfying $\limsup \nicefrac{p(q)}{q} < 1.5 + \epsilon$.  For arbitrary $f(q) \to \infty$, said subshifts can be made to satisfy $p(q) < q + f(q)$ infinitely often.

We establish that every subshift associated to a rank-one transformation (on a probability space) which is not an odometer satisfies $\limsup p(q) - 1.5q = \infty$ and that this is optimal for rank-ones.
}

\makepreprint

\dsectionstar{Introduction}

Hedlund and Morse \cite{morse1938symbolic} initiated the study of word complexity of symbolic systems: given an infinite word $x \in \mathcal{A}^{\mathbb{Z}}$, on some finite set $\mathcal{A}$--the alphabet, the word complexity $p(q)$ is the number of distinct subwords of $x$ of length $q$; more generally for a closed, shift-invariant $X \subseteq \mathcal{A}^{\mathbb{Z}}$, i.e.~a subshift, the complexity $p(q)$ is the number of distinct subwords of length $q$ appearing in any of the $x \in X$.

\cite{morse1940symbolic} established the first lower bound on the word complexity in terms of the structure of the subshift: if $x$ is aperiodic then $p(q) \geq q + 1$ for all $q$. 
A natural question, considering aperiodicity to be a weak form of mixing-like behavior, is to what extent mixing-type properties impose lower bounds on complexity, especially in light of recent results (e.g.~\cite{cassaigneetal}, \cite{CK2}, \cite{CK3}, \cite{CK4}, \cite{DDMP}, \cite{DOP}, \cite{ormespavlov}, \cite{PS}) regarding subshifts with low word complexity being highly structured.

\cite{morse1940symbolic} also exhibited words with $p(q) = q + 1$, called Sturmian words, which can be encoded by irrational rotations \cite{CH73}.  As irrational rotations are totally ergodic, the natural question is whether weak mixing imposes any sort of stronger lower bound on word complexity.
Topological mixing properties were considered by Gao and Ziegler \cite{gaoziegler} (see also Gao and Hill \cite{gaohillbounded}, \cite{gaohilltopo}); here we address the measure-theoretic question.  

The lowest previously known complexity for a subshift admitting a weakly mixing (probability) measure, due to
Ferenczi \cite{ferenczichacon}, is a subshift with complexity satisfying $\limsup \nicefrac{p(q)}{q} = \nicefrac{5}{3}$ and $\liminf \nicefrac{p(q)}{q} = 1.5$.
We exhibit subshifts, admitting weakly mixing (probability) measures, with lower complexity:

\begin{maintheorem}[Theorem \ref{Areal}]\label{A}
For every $\epsilon > 0$, there exists a weakly mixing rank-one transformation (on a probability space) such that the associated subshift has complexity $\limsup \nicefrac{p(q)}{q} < 1.5 + \epsilon$.
\end{maintheorem}

\begin{maintheorem}[Theorem \ref{Areal}]\label{A2}
For any $f(q) \to \infty$, the subshifts can be made to satisfy $p(q) < q + f(q)$ infinitely often.
\end{maintheorem}

Naturally, one wonders whether these bounds are sharp.  Cassaigne \cite{Cassaigne98sequenceswith} showed that if $p(q) = q + c$ for some constant $c$ then it is the image of a Sturmian word (so cannot admit a weakly mixing measure); this implies $p(q) < q + f(q)$ infinitely often is the best possible (see Proposition \ref{io} for specifics).

The analogous question for strong mixing was first explored by Ferenczi \cite{ferenczi1996rank} who showed that the classical staircase transformation (proved mixing by Adams \cite{adams1998smorodinsky}) has quadratic complexity and conjectured that was the minimal possible.  The author, Pavlov and Rodock \cite{CPR} disproved this conjecture; recently the author \cite{Creutzmixing} showed that strong mixing manifests exactly at superlinear complexity: every strongly mixing subshift satisfies $\lim \nicefrac{p(q)}{q} = \infty$ and for any $f(q) \to \infty$ there exist strongly mixing subshifts with $\lim \nicefrac{p(q)}{(qf(q))} = 0$.

We establish that $\limsup \nicefrac{p(q)}{q} = 1.5$ is optimal for rank-one transformations:

\begin{maintheorem}[Theorem \ref{mainreal}]\label{main}
Let $T$ be a rank-one transformation (on a probability space) which is not an odometer.  Then the associated subshift has complexity satisfying $\limsup p(q) - 1.5q = \infty$ (and $\liminf p(q) - q = \infty$).
\end{maintheorem}

While Sturmian words are encoded by irrational rotations (which are totally ergodic and rank-one), Rote \cite{rote} showed that the general word encoded by an irrational rotation has complexity $p(q) = 2q$ so if one treats an irrational rotation as a rank-one subshift then the complexity satisfies $p(q) \geq 2q$.

There appears to be a complexity distinction between totally ergodic and weakly mixing rank-one subshifts, namely that we can exhibit examples of totally ergodic rank-one subshifts with strictly lower complexity than any of our weakly mixing examples.  Specifically, Theorem \ref{main} is optimal:

\begin{maintheorem}[Theorem \ref{Dreal}]\label{te1.5}
For every $f(q) \to \infty$, there exists a totally ergodic rank-one transformation (on a probability space) such that the associated subshift satisfies $p(q) < 1.5q + f(q)$ for all sufficiently large $q$ and $p(q) < q + f(q)$ infinitely often.
\end{maintheorem}

It is worth remarking that $\limsup p(q)-1.5q = \infty$ distinguishing behavior in subshifts also appears in the work of Ormes and Pavlov \cite{ormespavlov} who showed that if $\limsup p(q) - 1.5q < \infty$ then the words in question are necessarily uniformly recurrent or bidirectionally eventually periodic.  For rank-one transformations, having bounded spacers implies uniform recurrence so their result and ours do not meaningfully overlap.  However, it is interesting that $\limsup p(q) - 1.5q < \infty$ is exactly the bound that rules out total ergodicity for rank-one subshifts as it is well-known that the lack of total ergodicity for rank-ones is equivalent to factoring onto a finite cyclic permutation, which is similar in spirit to their conclusion.

In connection with other properties often discussed with rank-one transformations, if we replace $p(q) < q + f(q)$ infinitely often with a slightly weaker condition then work of Ryzhikov \cite{ryzhikovmsj} gives:
\begin{maintheorem}[Theorem \ref{msjreal}]\label{msj}
For every $\epsilon > 0$, there exists a subshift with complexity satisfying
$\limsup \nicefrac{p(q)}{q} < 1.5 + \epsilon$ and $\liminf \nicefrac{p(q)}{q} < 1 + \epsilon$
such that the associated rank-one transformation is weakly mixing (on a probability space) and has minimal self-joinings (hence also has trivial centralizer and is mildly mixing).
\end{maintheorem}

The proof of Theorem \ref{main} is worth outlining briefly.  First we establish that for a rank-one subshift with $\limsup \nicefrac{p(q)}{q} < 2$, there is a rank-one subshift which generates the same language such that the spacer sequence eventually takes on at most two values.  Not being an odometer implies that both values must occur infinitely often and one can arrange for both to occur at every level (this arranging can lead to the cut sequence growing very rapidly).

The proof then proceeds by an analysis of all possible rank-one subshifts with exactly two spacer values.  We remark that finding our low complexity examples was a direct result of this examination, which both indicated $1.5$ ought to be the optimal bound and led to which subshifts were the correct candidates.

There remain questions regarding the precise nature of the complexity of subshifts admitting weakly mixing measures; we discuss these in Section \ref{questions}.  The main question left open is whether there exists a subshift, necessarily not rank-one, admitting a weakly mixing (probability) measure such that $\limsup \nicefrac{p(q)}{q} < 1.5$.  We tentatively conjecture that this is not the case and a bit more: for every subshift admitting a weakly mixing (probability) measure, we tentatively conjecture that $\limsup \nicefrac{p(q)}{q} > 1.5$.

Section \ref{S4} where the examples are constructed (and Section \ref{Swm} where weak mixing is proved) may be read independently; the reader primarily interested in the examples may opt to skip Sections \ref{S2} and \ref{S3} which are aimed at proving Theorem \ref{main}.

\textbf{Acknowledgments} The author would like to thank the referee for suggesting several welcome improvements to the exposition and to thank R.~Pavlov for a discussion that prompted the realization that the results hold for all non-odometer rank-ones (rather than only for totally ergodic rank-ones as stated in earlier drafts).

\section{Definitions and preliminaries}\label{S1}

\subsection{Symbolic dynamics}

\begin{definition}
A \textbf{subshift} on the finite set $\mathcal{A}$ is any subset $X \subset \mathcal{A}^{\mathbb{Z}}$ which is closed in the product topology and shift-invariant: for $x = (x_{n})_{n \in \mathbb{Z}} \in X$ and $k \in \mathbb{Z}$, the translate $(x_{n+k})_{n \in \mathbb{Z}}$ is also in $X$.
\end{definition} 

\begin{definition}
A \textbf{word} is any element of $\mathcal{A}^\ell$ for some $\ell$, the \textbf{length} of $w$, written $\len(w)$.  A word $w$ is a \textbf{subword} of a word or biinfinite sequence $x$ if there exists $k$ so that $w_{i} = x_{i+k}$ for $1 \leq i \leq \len(w)$.  A word $u$ is a \textbf{prefix} of $w$ if $u_{i} = w_{i}$ for $1 \leq i \leq \len(u)$ and a word $v$ is a \textbf{suffix} of $w$ if $v_{i} = w_{i+\len(w)-\len(v)}$ for $1 \leq i \leq \len(v)$.  A subword (or prefix or suffix) is \textbf{proper} when it is not the entire word.
\end{definition}

For words $v,w$, we denote by $vw$ their concatenation--the word obtained by following $v$ immediately by $w$. We also write such concatenations with product or exponential notation, e.g.~$\prod_i w_i$ or $0^n$.

\begin{definition}
The \textbf{language} of a subshift $X$ is $\mathcal{L}(X) = \{ w : \text{$w$ is a subword of some $x \in X$} \}$.
\end{definition}

\begin{definition}
The \textbf{word complexity function} of a subshift $X$ over $\mathcal{A}$ is the function $p_X: \mathbb{N} \rightarrow \mathbb{N}$ defined by $p_X(q) = |\mathcal{L}(X) \cap \mathcal{A}^q|$, the number of words of length $q$ in the language of $X$.
\end{definition}

When $X$ is clear from context, we suppress the subscript and just write $p(q)$.

\subsubsection{Right-special words}

All subshifts we consider are on the alphabet $\{0,1\}$ so it is natural to consider:
\begin{definition}
The set of \textbf{right-special} words is
$
\mathcal{L}^{RS}(X) = \{ w \in X : w0, w1 \in \mathcal{L}(X) \}
$.
\end{definition}

Cassaigne \cite{cassaigne} showed the well-known: 
$
p(q) = p(m) + \sum_{\ell = m}^{q-1} |\{ w \in \mathcal{L}^{RS} : \len(w) = \ell \}|
$
for $m < q$.

\subsubsection{Quasi-Sturmian words}

An infinite $x \in \mathcal{A}^{\mathbb{N}}$ is \textbf{Sturmian} when $p_x(q) = q + p_x(1)$.  Hedlund and Morse \cite{morse1940symbolic} exhibited examples of such words and showed that if $p_x(q) \leq q$ or $p_x(q+1) = p_x(q)$ for any $q$ then $x$ is periodic.

Cassaigne \cite{Cassaigne98sequenceswith} termed infinite words $x$ such that $p_x(q) = q + c$ for some constant $c$ and all sufficiently large $q$ \textbf{quasi-Sturmian} and showed such a word must be the image of a Sturmian word under a morphism $f : \mathcal{A}^{*} \to \mathcal{A}^{*}$ which is nonperiodic.

Indeed, his result quickly gives a bit more:
\begin{proposition}\label{io}
Let $X$ be an aperiodic subshift such that $p_{X}(q) \leq q + d$ for some constant $d$ and infinitely many $q$.  Then $X$ is quasi-Sturmian (in the sense that all $x \in X$ are quasi-Sturmian) hence cannot admit a weakly mixing measure.
\end{proposition}
\begin{proof}
By the Hedlund-Morse theorem, we may assume $p(\ell+1) - p(\ell) \geq 1$ for all $\ell$ since otherwise the subshift is periodic.  For infinitely many $q$,
\[
q + d \geq p(q) = p(1) + \sum_{\ell=1}^{q-1} (p(\ell+1)-p(\ell)) \geq p(1) + q - 1 + |\{ \ell < q : p(\ell+1) - p(\ell) \geq 2 \}|
\]
so for infinitely many $q$ we have $|\{ \ell < q : p(\ell+1) - p(\ell) \geq 2 \}| < d$ meaning $|\{ \ell : p(\ell+1) - p(\ell) \geq 2 \}| < d$.

Set $c = p(1) - 1 + \sum_{\ell=1}^{\infty} (p(\ell+1) - p(\ell) - 1)$, which must be finite as there are only finitely many $\ell$ with $p(\ell + 1) - p(\ell) > 1$.  Then for all $q > \max\{ \ell : p(\ell+1) - p(\ell) \geq 2\}$,
\[
p(q) = p(1) + q - 1 + \sum_{\ell=1}^{q-1} (p(\ell+1) - p(\ell) - 1)
=  q + c
\]
Since Sturmian words can be encoded by irrational rotations, Sturmian (and therefore quasi-Sturmian) subshifts cannot admit weakly mixing measures.
\end{proof}

\subsection{Ergodic theory}

\begin{definition}
A \textbf{transformation} $T$ is a measurable map on a standard Borel or Lebesgue measure space $(Y,\mathcal{B},\mu)$ that is measure-preserving: $\mu(T^{-1}B) = \mu(B)$ for all $B \in \mathcal{B}$.
\end{definition}

\begin{definition}Two transformations $T$ on $(Y,\mathcal{B},\mu)$ and $T^{\prime}$ on $(Y^{\prime},\mathcal{B}^{\prime},\mu^{\prime})$ are \textbf{measure-theoretically isomorphic} if there exists 
a bijective map $\phi$ between full measure subsets $Y_0 \subset Y$ and $Y'_0 \subset Y'$ where 
$\mu(\phi^{-1} A) = \mu'(A)$ for all measurable $A \subset Y'_0$ and $(\phi \circ T)(y) = (T' \circ \phi)(y)$ for all $y \in Y_0$.
\end{definition}

\begin{definition}
A transformation $T$ is \textbf{ergodic} when $A = T^{-1} A$ implies that $\mu(A) = 0$ or $\mu(A^c) = 0$. 
\end{definition}

\begin{definition}
A transformation $T$ is \textbf{totally ergodic} when $T^{k}$ is ergodic for all $k \in \mathbb{N}$.
\end{definition}

\begin{definition}
A transformation $T$ on a probability space is \textbf{weakly mixing} when any of the following equivalent conditions hold:
\begin{itemize}
\item for all measurable sets $A,B$, there exists $\{ t_{n} \}$ such that $\mu(T^{t_{n}}A \cap B) \to \mu(A)\mu(B)$
\item there exists a density one $\{ t_{n} \}$ such that $\mu(T^{t_{n}}A \cap B) \to \mu(A)\mu(B)$ for all measurable sets $A,B$
\item $T \times T$ is ergodic
\item for all measurable $A,B$ there exists $n$ such that $\mu(T^{n}A\cap A)\mu(T^{n}A \cap B) > 0$
\end{itemize}
\end{definition}

\subsection{Rank-one transformations}

A \textbf{rank-one transformation} is a transformation $T$ constructed by ``cutting and stacking".  Here $Y$ represents a (possibly infinite) interval, $\mathcal{B}$ is the induced $\sigma$-algebra from $\mathbb{R}$, and $\mu$ is Lebesgue measure. We give a brief description, referring the reader to 
\cite{fghsw21} or  \cite{silva2008invitation} for more details and to \cite{ferenczisurv} for equivalent definitions.

The transformation is defined inductively on larger and larger portions of the space through Rohlin towers or \textbf{columns}, denoted $C_n$. Each column $C_n$ consists of \textbf{levels} $I_{n,j}$ where $0 \leq j < h_{n}$ is the height of the level within the column. All levels $I_{n,j}$ in $C_n$ are 
intervals with the same length, $\mu(I_{ n})$, and the total number of levels in a column is the \textbf{height} of the column, denoted by $h_n$. The transformation $T$ is defined on all levels $I_{n,j}$ except the top one $I_{n, h_n - 1}$ by sending each $I_{n,j}$ to $I_{n,j+1}$ using the unique order-preserving affine map.

Start with $C_1=[0,1)$ with height $h_1=1$. To obtain $C_{n+1}$ from $C_n$, we require a \textbf{cut sequence}, $\{r_n\}$ such that $r_n \geq 1$ for all $n$. Make $r_n$ vertical cuts of $C_n$ to create $r_n+1$ \textbf{subcolumns} of equal width. Denote a \textbf{sublevel} of $C_n$ by $I_{n,j}^{[i]}$ where $0 \leq a < h_{n}$ is the height of the level within that column, and 
$i$ represents the position of the subcolumn, where $i=0$ represents the leftmost subcolumn and $i=r_n$ is the rightmost subcolumn. After cutting $C_n$ into subcolumns, add extra intervals called \textbf{spacers} on top of each subcolumn to function as levels of the next column. The \textbf{spacer sequence}, $\{s_{n,i}\}$ such that $0 \leq i \leq r_{n}$ and $s_{n,i} \geq 0$, specifies how many sublevels to add above each subcolumn. Spacers are the same width as the sublevels, act as new levels in the column $C_{n+1}$, and are taken to be the leftmost intervals in $[1,\infty)$ not in $C_{n}$.
After the spacers are added, stack the subcolumns with their spacers right on top of left, i.e.~so that $I_{n,0}^{[i+1]}$ is directly above $I_{n,h_{n}-1}^{[i]}$. This gives the next column, $C_{n+1}$.

Each column $C_n$ defines $T$ on $\bigcup_{j = 0}^{h_n - 2} I_{n,j}$ and 
the partially defined map $T$ on $C_{n+1}$ agrees with that of $C_n$, extending the definition of $T$ to a portion of the top level of $C_n$ where it was previously undefined. Continuing this process gives the sequence of columns $\{C_1, \dots, C_n, C_{n+1}, \dots\}$ and $T$ is then the limit of the partially defined maps. 

Though this construction could result in $Y$ being an infinite interval with infinite Lebesgue measure, $Y$ has finite measure if and only if 
$\sum_{n} \frac{1}{r_{n}h_{n}}\sum_{i=0}^{r_{n}} s_{n,i} < \infty$, see \cite{CreutzSilva2010}.
All rank-one transformations we define satisfy this condition, and for convenience we renormalize so that $Y = [0,1)$. Every rank-one transformation is ergodic and invertible.

The reader should be aware that we are making $r_{n}$ cuts and obtaining $r_{n}+1$ subcolumns (following Ferenczi \cite{ferenczi1996rank}), while other papers (e.g.~\cite{Creutz2021}) use $r_{n}$ as the number of subcolumns.

\subsection{Odometers}

\begin{definition}
A rank-one transformation which can be constructed using a spacer sequence such that there exists $N$ so that $s_{n,i} = 0$ for all $n\geq N$ and $0\leq i < r_{n}$ is an \textbf{odometer}.
\end{definition}

Odometers have discrete spectrum and all their eigenvalues are rational in the sense that they are of the form $\exp(2\pi i q)$ for $q \in \mathbb{Q}$.

\subsection{Symbolic models of rank-one transformations}

For a rank-one transformation defined as above, we define a subshift $X(T)$ on the alphabet $\{ 0,1 \}$ which is measure-theoretically isomorphic to $T$:

\begin{definition}
The \textbf{symbolic model} $X(T)$ of, or \textbf{subshift associated} to, a rank-one transformation $T$ is given by the sequence of words: $B_1 = 0$ and 
\begin{equation*}\label{rk1word}
B_{n+1}=B_n1^{s_{n,0}}B_n1^{s_{n,1}}\cdots B_{n} 1^{s_{n,r_n}}= \prod_{i=0}^{r_n} B_n1^{s_{n,i}}
\end{equation*}
and $X(T)$ is the set of all biinfinite sequences such that every subword is a subword of some $B_n$.
\end{definition}

The words $B_n$ are a symbolic coding of the column $C_n$: $0$ represents $C_1$ and $1$ represents the spacers, and $h_{n} = \len(B_{n})$.  There is a natural measure associated to $X(T)$:

\begin{definition}
The \textbf{empirical measure} for a symbolic model $X(T)$ of a rank-one transformation $T$ is the measure $\nu$ defined by, for each word $w$,
\[
\nu([w]) = \lim_{n \rightarrow \infty} \frac{|\{1 \leq j \leq \len(B_{n}) - \len(w) \ : \ B_n[j; \len(w)] = w\}|}{\len(B_{n}) - \len(w)}
\]
where $B_{n}[j; \ell]$ denotes the subword of $B_{n}$ starting at position $j$ with length $\ell$.
\end{definition}

Danilenko \cite{danilenko16} (combined with \cite{deljunco} and \cite{kalikow}) proved that the symbolic model $X(T)$ of a rank-one subshift, equipped with its empirical measure, is measure-theoretically isomorphic to the cut-and-stack construction (see \cite{adamsferenczipeterson17}; see \cite{fghsw21} for the full generality including odometers).

Due to this isomorphism, we move back and forth between rank-one and symbolic model terminology as needed and write $\mathcal{L}(T)$ for the language of $X(T)$, or simply $\mathcal{L}$ if $X(T)$ is clear from context, and define:

\begin{definition}
A \textbf{rank-one subshift} is the symbolic model of a rank-one transformation.
\end{definition}

Likewise, when the measure is clear from text, such as the empirical measure for a rank-one subshift:
\begin{definition}
A \textbf{(measure-theoretically) weakly mixing subshift} is a subshift for which the measure is weakly mixing.
\end{definition}

\section{Properties of rank-one subshifts}\label{S2}

\begin{lemma}\label{obv}
For $n < m$, $B_{m}$ has $B_{n}$ as a prefix and $B_{n}1^{s_{n,r_{n}}}$ as a suffix.
\end{lemma}
\begin{proof}
This is immediate from the construction.
\end{proof}

\begin{lemma}\label{0}
$B_{n}$ has $0$ as a prefix for all $n$.
\end{lemma}
\begin{proof}
By Lemma \ref{obv}, $B_{n}$ has $B_{1} = 0$ as a prefix.
\end{proof}

We next need a result of Danilenko:
\begin{proposition}[\cite{MR4041118} Lemma 1.10]\label{1}
Every rank-one subshift is measure-theoretically isomorphic to a rank-one subshift with $s_{n,r_{n}} = 0$ and the two subshifts generate the same language.
\end{proposition}
%

\begin{proposition}\label{finmeasinf}
For  a rank-one subshift on a finite measure space,
$
\frac{1}{h_{n}} \inf \{ s_{n,i} : 0 \leq i < r_{n} \} \to 0
$.
\end{proposition}
\begin{proof}
Suppose $\inf \{ s_{n,i} : 0 \leq i < r_{n} \} \geq \delta h_{n}$ infinitely often for some $\delta > 0$.  Then for such $n$, we have $\mu(C_{n+1}) \geq \mu(C_{n}) + \inf \{ s_{n,i} : 0 \leq i < r_{n} \}\mu(I_{n}) \geq (1 + \delta)\mu(C_{n})$.  So for any $k$, if we choose $N$ such that at least $k$ values of $n < N$ have $\inf \{ s_{n,i} : 0 \leq i < r_{n} \} \geq \delta h_{n}$ then $\mu(C_{N}) \geq (1 + \delta)^{k} \mu(C_{0})$.  Taking $k \to \infty$ shows the measure would then be infinite.
\end{proof}

\subsection{Rank-one subshifts with at least three distinct spacer values}

\begin{proposition}\label{6}
For a rank-one subshift on a finite measure space with $s_{n,r_{n}} = 0$ for all sufficiently large $n$, if the set $\{ s_{m,i} : m \geq n, 0 \leq i < r_{m} \}$ contains at least three distinct values for infinitely many $n$ then $\limsup \frac{p(q)}{q} \geq 2$.
\end{proposition}
\begin{proof}
Choose $n$ such that $t_{n} = \inf \{ s_{n,i} : 0 \leq i < r_{n} \}$ has the property that $t_{n} = \inf \{ s_{m,i} : m \geq n, 0 \leq i < r_{m} \}$ (such an $n$ must exist since otherwise there is a sequence $\{ m_{t} \}$ along which $\inf \{ s_{m_{t},i} : 0 \leq i < r_{m_{t}} \}$ is strictly decreasing which would contradict that $s_{m,i} \geq 0$).

Let $u_{n}, v_{n} \in \{ s_{m,i} : m \geq n, 0 \leq i < r_{m} \}$ such that $t_{n} < u_{n} < v_{n}$.  Such must exist since otherwise $|\{ s_{m,i} : m \geq n, 0 \leq i < r_{m^{\prime}} \}| = 2$ so the same holds for all $n^{\prime} \geq n$.

The word $B_{n}1^{t_{n}}B_{n}$ is a subword of $B_{n+1}$.  As $B_{n}$ has $0$ as a prefix, $B_{n}1^{t_{n}}0 \in \mathcal{L}$.  As $u_{n} > t_{n}$ and $B_{n}1^{u_{n}}$ is a subword of $B_{m}1^{u_{n}}$ which is a subword of $B_{m + 1}$, this shows $B_{n}1^{t_{n}} \in \mathcal{L}^{RS}$.  Likewise $B_{n}1^{u_{n}} \in \mathcal{L}^{RS}$.

Let $N$ such that $s_{n,r_{n}} = 0$ for $n \geq N$.  Let $c \geq 1$ such that $B_{N}$ has $01^{c-1}$ as a suffix (such $c \leq h_{N}$ must exist as $B_{N}$ has $0$ as a prefix).  Since $s_{n,r_{n}}=0$ for $n \geq N$, the word $B_{n}$, for all $n \geq N$, has $B_{N}$ as a suffix hence has $01^{c-1}$ as a suffix.

Therefore $B_{n}1^{t_{n}}$ has $01^{c-1+t_{n}}$ as a suffix and $B_{n}1^{u_{n}}$ has $01^{c-1+u_{n}}$ as a suffix meaning that for every $t_{n} + c \leq \ell < h_{n} + t_{n}$, the suffixes of $B_{n}1^{t_{n}}$ and $B_{n}1^{u_{n}}$ of length $\ell$ are distinct (as $u_{n} > t_{n}$).

Then $p(\ell+1) - p(\ell) = |\{ w \in \mathcal{L}^{RS} : \len(w) = \ell \}| \geq 2$ for $t_{n} + c \leq \ell < h_{n} + t_{n}$ meaning that $p(h_{n}) \geq 2(h_{n} - t_{n} - c)$ so, as Proposition \ref{finmeasinf} implies $\frac{t_{n}}{h_{n}} \to 0$,
$
\frac{p(h_{n})}{h_{n}} \geq 2(1 - \frac{t_{n}+c}{h_{n}}) \to 2
$.
\end{proof}

\subsection{Rank-one subshifts with the same language}

\begin{lemma}\label{trick1a}
Let $T$ be a rank-one subshift with cut sequence $\{ r_{n} \}$ and spacer sequence $\{ s_{n,i} \}$.

Let $N \in \mathbb{N}$.
For $n < N$, set $\tilde{r}_{n} = r_{n}$ and $\tilde{s}_{n,i} = s_{n,i}$.

Set $\tilde{r}_{N} = (r_{N} + 1)(r_{N+1}+1) - 1$ and for $0 \leq a \leq r_{N+1}$, set $\tilde{s}_{N,a(r_{N}+1) + b} = s_{N,b}$ for $0 \leq b < r_{N}$ and set $\tilde{s}_{N,a(r_{N}+1) + r_{N}} = s_{N,r_{N}} + s_{N+1,a}$.

For $n > N$, set $\tilde{r}_{n} = r_{n+1}$ and $\tilde{s}_{n,i} = s_{n+1,i}$.

Then the rank-one subshift $\tilde{T}$ generates the same language as $T$.
\end{lemma}
\begin{proof}
Clearly $\tilde{B}_{n} = B_{n}$ for $n \leq N$.  By design,
\begin{align*}
\tilde{B}_{N+1} &= \prod_{a=0}^{r_{N+1}}\big{(}\prod_{b=0}^{r_{N}}\tilde{B}_{N}1^{\tilde{s}_{N,a(r_{N}+1)+b}}\big{)} \\
&= \prod_{a=0}^{r_{N+1}}\Big{(}\big{(}\prod_{b=0}^{r_{N}-1}B_{N}1^{s_{N,b}}\big{)}B_{N}1^{s_{N,r_{N}}+s_{N+1,a}}\Big{)} 
= \prod_{a=0}^{r_{N+1}}B_{N+1}1^{s_{N+1,a}} = B_{N+2}
\end{align*}
so $\tilde{B}_{n} = B_{n+1}$ for all $n > N$.
\end{proof}

\begin{proposition}\label{trick}
Let $T$ be a rank-one transformation such that $s_{n,r_{n}} = 0$ and $s_{n,i} = c_{n}$ for $0 \leq i < r_{n}$ for all sufficiently large $n$.  If $c_{n}$ is not eventually constant then there exists a rank-one subshift $\tilde{T}$ which generates the same language as $T$, with the property that $\tilde{s}_{n,\tilde{r}_{n}} = 0$ and $\tilde{s}_{n,i}$ is not constant over $0 \leq i < \tilde{r}_{n}$ for infinitely many $n$.
\end{proposition}
\begin{proof}
If $c_{n}$ is not eventually constant then there exist infinitely many $n < m$ such that $c_{n} \ne c_{m}$ so there exist infinitely many $n$ such that $c_{n} \ne c_{n+1}$.

If we apply Lemma \ref{trick1a} at such an $n$ then $\tilde{s}_{n,i}$ is not constant over $0 \leq i < \tilde{r}_{n}$ since $\tilde{s}_{n,r_{n}} = s_{n,r_{n}} + s_{n+1,0} = 0 + c_{n+1} \ne c_{n} = \tilde{s}_{n,0}$ and $\tilde{s}_{n,\tilde{r}_{n}} = s_{n,r_{n}} + s_{n+1,r_{n+1}} = 0$.

Let $\mathcal{N}$ be a set of $n$ such that $c_{n} \ne c_{n+1}$ such that $\mathcal{N}$ does not contain any pairs of consecutive integers.  Applying Lemma \ref{trick1a} for each $n \in \mathcal{N}$ gives the claim.
\end{proof}

In fact, one can do a similar modification across multiple stages simultaneously:

\begin{lemma}\label{trick1}
Let $T$ be a rank-one subshift with cut sequence $\{ r_{n} \}$ and spacer sequence $\{ s_{n,i} \}$ and let $\{ n_{t} \}$ be a strictly increasing sequence with $n_{1} = 1$.  For $t \geq 1$, set
\[
\tilde{r}_{t} = \Big{(}\prod_{n=n_{t}}^{n_{t+1}-1} (r_{n}+1)\Big{)} - 1
\]
and, for $0 \leq j < n_{t+1} - n_{t}$ and $0 \leq i_{j} \leq r_{n_{t}+j}$,
\begin{align*}
&\tilde{s}_{t,i_{0} + i_{1}(r_{n_{t}}+1) + i_{2}(r_{n_{t}+1}+1)(r_{n_{t}}+1) + \cdots + i_{n_{t+1}-n_{t}-1}(r_{n_{t+1}-1} + 1)\cdots (r_{n_{t}}+1)} \\
&\quad\quad\quad = s_{n_{t},i_{0}} + \sum_{j=1}^{n_{t+1}-n_{t}-1} \left\{ \begin{array}{ll} s_{n_{t}+j,i_{j}} \quad &\text{if $i_{k} = r_{n_{t}+k}$ for all $0 \leq k < j$} \\ 0 & \text{otherwise} \end{array} \right.
\end{align*}
Then $T$ and $\tilde{T}$ generate the same language: $\tilde{B}_{t} = B_{n_{t}}$ for all $t \geq 1$.
\end{lemma}
\begin{proof}
We have $\tilde{B}_{1} = 0 = B_{1} = B_{n_{1}}$, so we may assume $\tilde{B}_{t} = B_{n_{t}}$ and then
\begin{align*}
\tilde{B}_{t+1} = \prod_{a=0}^{\tilde{r}_{t}} \tilde{B}_{t}1^{\tilde{s}_{t,a}} 
&= \prod_{i_{0},\ldots,i_{n_{t+1}-n_{t}-1}} B_{n_{t}}1^{s_{n_{t},i_{0}}}1^{\sum_{j=1}^{n_{t+1}-n_{t}-1} s_{n_{t}+j,i_{j}}\bbone_{i_{k}=r_{n_{t}+k}\forall k<j}} \\
&= \prod_{i_{1},\ldots,i_{n_{t+1}-n_{t}-1}} \big{(}\prod_{i_{0}=0}^{r_{n_{t}}} B_{n_{t}}1^{s_{n_{t},i_{0}}}\big{)} 1^{s_{n_{t}+1,i_{1}}} 1^{\sum_{j=2}^{n_{t+1}-n_{t}-1} s_{n_{t}+j,i_{j}}\bbone_{i_{k}=r_{n_{t}+k}\forall k< j}} \\
&= \prod_{i_{1},\ldots,i_{n_{t+1}-n_{t}-1}} B_{n_{t}+1} 1^{s_{n_{t}+1,i_{1}}} 1^{\sum_{j=2}^{n_{t+1}-n_{t}-1}s_{n_{t}+j,i_{j}}\bbone_{i_{k}=r_{n_{t}+k}\forall k< j}} \\
&= \prod_{i_{2},\ldots,i_{n_{t+1}-n_{t}-1}} \big{(}\prod_{i_{1}=0}^{r_{n_{t}+1}} B_{n_{t}+1} 1^{s_{n_{t}+1,i_{1}}}\big{)} 1^{s_{n_{t}+2,i_{2}}} 1^{\sum_{j=3}^{n_{t+1}-n_{t}-1}s_{n_{t}+j,i_{j}}\bbone_{i_{k}=r_{n_{t}+k}\forall k< j}} \\
&\quad\quad\quad\quad\quad\quad\quad\quad\quad\quad\quad\quad\quad \vdots \\
&= \prod_{i_{n_{t+1}-n_{t}-1}=0}^{r_{n_{t+1}-n_{t}-1}} B_{n_{t+1}-1}1^{s_{n_{t+1}-1,i_{n_{t+1}-n_{t}-1}}}
= B_{n_{t+1}} \qedhere
\end{align*}
\end{proof}

\begin{proposition}\label{trick2}
Let $T$ be a rank-one subshift such that $s_{n,r_{n}} = 0$ for all sufficiently large $n$ and that there exists $0 \leq i, i^{\prime} < r_{n}$ such that $s_{n,i} \ne s_{n,i^{\prime}}$ for infinitely many $n$ .  Then there exists a rank-one subshift $\tilde{T}$, which generates the same language, such that for all sufficiently large $n$, $\tilde{s}_{n,\tilde{r}_{n}} = 0$ and there exists $0 \leq i,i^{\prime} < \tilde{r}_{n}$ with $\tilde{s}_{n,i} \ne \tilde{s}_{n,i^{\prime}}$.
\end{proposition}
\begin{proof}
Let $n_{1} = 1$ and $\{ n_{t} \}_{t \geq 2}$ be the sequence of $n$ for which $s_{n,i} \ne s_{n,i^{\prime}}$.  Lemma \ref{trick1} then gives the claim since $s_{n_{t},i}$ being nonconstant over $0 \leq i < r_{n_{t}}$ implies $\tilde{s}_{t,a}$ is nonconstant over $0 \leq a < r_{n_{t}}$ hence over $0 \leq a < \tilde{r}_{t}$.  Clearly $\tilde{s}_{t,\tilde{r}_{t}} = 0$ for sufficiently large $t$ as $s_{n,r_{n}} = 0$ for all sufficiently large $n$.
\end{proof}

\begin{proposition}\label{trick3}
Let $T$ be a rank-one subshift such that $s_{n,r_{n}} = 0$ for all sufficiently large $n$ and that for infinitely many $n$,  $s_{n,0} = s_{n,r_{n}-1} = 0$.  Then there exists a rank-one subshift, which generates the same language, such that $\tilde{s}_{n,r_{n}} = 0$ and $\tilde{s}_{n,0} = \tilde{s}_{n,\tilde{r}_{n}-1} = \tilde{s}_{n+1,\tilde{r}_{n+1}-1} = 0$ for all sufficiently large $n$.
\end{proposition}
\begin{proof}
Let $n_{1} = 1$ and $\{ n_{t} \}_{t \geq 2}$ be the sequence of $n$ for which $s_{n,0} = s_{n,r_{n}-1} = 0$.  Lemma \ref{trick1} then gives the subshift since $\tilde{s}_{t,0} = s_{n_{t},0} = 0$ and $\tilde{s}_{t,\tilde{r}_{t}-1}$ has $i_{0} = r_{n_{t}}-1$ so $\tilde{s}_{t,\tilde{r}_{t}-1} = s_{n_{t},r_{n_{t}}-1} = 0$ and, likewise, $\tilde{s}_{t+1,\tilde{r}_{t+1}-1} = s_{n_{t+1},r_{n_{t+1}}-1} = 0$.
\end{proof}

\begin{proposition}\label{tcd}
If a rank-one subshift has the property that $s_{n,r_{n}} = 0$ for all sufficiently large $n$ and there exist constant nonnegative integers $c < d$ such that $s_{n,i} \in \{ c, d \}$ for all $0 \leq i < r_{n}$ (with both occurring) for sufficiently large $n$ then there exists a rank-one subshift which generates the same language such that $s_{n,r_{n}} = 0$ and $s_{n,i} \in \{ 0, d-c \}$ for all $0 \leq i < r_{n}$ (with both occurring) for all sufficiently large $n$.
\end{proposition}
\begin{proof}
For all $n$ set $\tilde{r}_{n} = r_{n}$.
Let $N$ such that for all $n \geq N$, we have $s_{n,i} \in \{ c, d \}$ for all $0 \leq i < r_{n}$ and $s_{n,r_{n}} = 0$.  For $n < N$, set $\tilde{s}_{n,i} = s_{n,i}$.

Set $\tilde{s}_{N,i} = s_{N,i}$ for $0 \leq i < r_{N}$ and $\tilde{s}_{N,r_{N}} = c$.
For $n > N$ set $\tilde{s}_{n,i} = s_{n,i} - c$ for $0 \leq i  < r_{n}$ and $\tilde{s}_{n,r_{n}} = 0$.

Clearly $\tilde{B}_{n} = B_{n}$ for $n \leq N$.  Observe that
\[
\tilde{B}_{N+1} = \big{(}\prod_{i=0}^{r_{N}-1}B_{N}1^{s_{N,i}}\big{)}B_{N}1^{c} = B_{N+1}1^{c}
\]
If $\tilde{B}_{n} = B_{n}1^{c}$ then
\[
\tilde{B}_{n+1} = \big{(}\prod_{i=0}^{r_{n}-1}\tilde{B}_{n}1^{\tilde{s}_{n,i}}\big{)}\tilde{B}_{n}
=  \big{(}\prod_{i=0}^{r_{n}-1}B_{n}1^{c} 1^{s_{n,i} - c}\big{)}B_{n}1^{c}
= B_{n+1}1^{c}
\]
so $\tilde{B}_{n} = B_{n}1^{c}$ for all $n > N$ meaning they generate the same language.
\end{proof}

\subsection{Totally ergodic rank-one subshifts}

%
%

\begin{proposition}\label{te}
Let $T$ be a rank-one transformation such that there exists $c$ so that for all sufficiently large $n$, it holds that $s_{n,i} = c$ for all $0 \leq i < r_{n}$ and $s_{n,r_{n}} = 0$.  Then $T$ is an odometer.
\end{proposition}
\begin{proof}
Let $N > 1$ such that for all $n \geq N$, $s_{n,i} = c$ for all $0 \leq i < r_{n}$ and $s_{n,r_{n}} = 0$.  Let $S_{n,j}^{[i]}$ for $1 \leq j \leq c$ be the spacer levels added above $C_{n}^{[i]}$ for $0 \leq i < r_{n}$ (we do not add spacers above $C_{n}^{[r_{n}]}$ as $s_{n,r_{n}} = 0$).
Since $T(S_{n,c}^{[i]}) = I_{n,0}^{[i+1]}$ for $0 \leq i < r_{n}$, and since $I_{n,0} \subseteq I_{N,0}^{[0]}$, we have that $T(S_{n,c}^{[i]}) \subseteq I_{N,0}^{[0]}$ for all $n \geq N$ and all $0 \leq i < r_{n}$.
Since $I_{N,h_{N}-1}^{[r_{N}]} = \bigsqcup_{n > N} \bigsqcup_{i=0}^{r_{n}-1} I_{n,h_{n}-1}^{[i]}$ this means $T^{h_{N}+c}(I_{N,0}) = I_{N,0}$.

Define $I_{N,h_{N}} = \bigsqcup_{n \geq N} \bigsqcup_{0 \leq i < r_{n}} S_{n,1}^{[i]}$.  Then $T(I_{N,h_{N}-1}) = I_{N,h_{N}}$ and $T^{c}(I_{N,h_{N}}) = I_{N,0}$.  Define the column $C_{N}^{\prime} = \bigsqcup_{j=0}^{h_{N}-1} I_{N,j} \sqcup \bigsqcup_{j=0}^{c-1} T^{j}(I_{N,h_{N}})$ and the columns $C_{N+n}^{\prime}$ via cutting and stacking starting from $C_{N}^{\prime}$ using cut sequence $r_{N+n}^{\prime} = r_{N+n}$ and spacer sequence $s_{n,i}^{\prime} = 0$.  The resulting odometer is the same map as $X$ so $X$ is an odometer.
\end{proof}

\begin{proposition}\label{2.3}
Let $T$ be a rank-one transformation on a finite measure space which is not an odometer.  If $\limsup p(q)/q < 2$ then
 there exists a rank-one subshift, which generates the same language as $T$, such that there exists a constant positive integer $d$ so that for all sufficiently large $n$ it holds that $s_{n,r_{n}} = 0$ and $s_{n,i} \in \{ 0, d \}$ for all $0 \leq i < r_{n}$ and there exists $0 \leq i, i^{\prime} < r_{n}$ so that $s_{n,i} = 0$ and $s_{n,i^{\prime}} = d$.
 \end{proposition}
\begin{proof}
By Proposition \ref{1}, $T$ is measure-theoretically isomorphic to a transformation $\tilde{T}$ which generates the same language and has $\tilde{s}_{n,\tilde{r}_{n}} = 0$ for all $n$.  By Proposition \ref{te}, $\tilde{T}$ has the property that for every $n$ and $0 \leq i < r_{n}$ there exists $m \geq n$ and $0 \leq i^{\prime} < r_{m}$ such that $s_{m,i} \ne s_{n,i^{\prime}}$.

By Proposition \ref{6}, if $\limsup_{N} |\{ s_{n,i} : n \geq N, 0 \leq i < r_{n} \}| \geq 3$ then $\limsup \frac{p(q)}{q} \geq 2$.  So there exists $N$ such that $|\{ s_{m,i} : m \geq N, 0 \leq i < r_{m} \}| \leq 2$.  Therefore $|\{ s_{m,i} : m \geq n, 0 \leq i < r_{m} \}| = 2$ for all sufficiently large $n$.

Proposition \ref{trick} gives a rank-one subshift generating the same language such that $s_{n,r_{n}} = 0$ for all sufficiently large $n$ and $s_{n,i} \ne s_{n,i^{\prime}}$ for infinitely many $n$.  Proposition \ref{trick2} then gives a rank-one subshift generating the same language with that property for all sufficiently large $n$.  Finally, Proposition \ref{tcd} gives a rank-one subshift, still generating the same language, such that $s_{n,i} \in \{ 0, d \}$ and $0 \leq i < r_{n}$ and $s_{n,r_{n}} = 0$ for all sufficiently large $n$.
\end{proof}

\section{Subshifts with exactly one nonzero spacer value}\label{S3}

\begin{theorem}\label{t}
Let $p$ be the complexity function for a rank-one subshift such that for all sufficiently large $n$,
the spacer sequence satisfies $s_{n,i} \in \{ 0, d \}$ for some constant positive integer $d$ and $s_{n,r_n}=0$ and that $s_{n,i}$ is not constant over $0 \leq i < r_n$.  Then $\limsup p(q) - 1.5q = \infty$.
\end{theorem}

This is a quick consequence of:
\begin{theorem}\label{t2}
Let $p$ be the complexity function for a rank-one subshift such that for all sufficiently large $n$,
the spacer sequence satisfies $s_{n,i} \in \{ 0, d \}$ for some constant positive integer $d$ and $s_{n,r_n}=0$ and that $s_{n,i}$ is not constant over $0 \leq i < r_n$.

Then there exists a constant $C$ such that for all sufficiently large $n$, there exists $q_{n} \geq h_{n}$ such that $p(q_{n}) \geq 1.5q_{n} + (p(h_{n}) - h_{n}) - C$.
\end{theorem}

\begin{proof}[Proof of Theorem \ref{t} from Theorem \ref{t2}]
Let $N$ such that for all $n \geq N$, there exists $q_{n} \geq h_{n}$ such that $p(q_{n}) \geq 1.5q_{n} + (p(h_{n})-h_{n}) - C$.  Let $m > n$ such that $h_{m} \geq q_{n}$.  As $s_{n,i} > 0$ for $i < r_{n}$ implies aperiodicity,
$p(\ell+1) - p(\ell) \geq 1$ for all $\ell$ so $p(h_{n}) \geq h_{n}$ and $p(h_{m}) - p(q_{n}) \geq h_{m} - q_{n}$.  Then
\begin{align*}
p(h_{m}) - h_{m} &= (p(h_{m}) - p(q_{n})) + p(q_{n}) - h_{m}
\geq (h_{m} - q_{n}) + 1.5q_{n} - C - h_{m}
= 0.5q_{n} - C \to \infty
\end{align*}
and therefore
$
p(q_{m}) - 1.5q_{m} \geq p(h_{m}) - h_{m} -  C \to \infty$.
\end{proof}

Before proving Theorem \ref{t2}, we show how Theorem \ref{t} implies:
\begin{theorem}\label{mainreal}
Let $T$ be a rank-one transformation (on a probability space) which is not an odometer.  Then the associated subshift has complexity satisfying $\limsup p(q) - 1.5q = \infty$ (and $\liminf p(q) - q = \infty$).
\end{theorem}
\begin{proof}
By Proposition \ref{2.3}, either $\limsup p(q)/q \geq 2$ or there exists a rank-one subshift which generates the same language with the property that there exists a constant nonnegative integer $d$ such that for all sufficiently large $n$, $s_{n,i} \in \{ 0, d \}$ for all $0 \leq i < r_{n}$ and $s_{n,r_{n}} = 0$ and such that there exists $0 \leq i,i^{\prime} < r_{n}$ with $s_{n,i} = 0$ and $s_{n,i^{\prime}} = d$.  Theorem \ref{t} applied to that subshift then gives that $\limsup p(q) - 1.5q = \infty$.  Proposition \ref{io} ensures $\liminf p(q) - q = \infty$ as otherwise $p(q) = q + c$ for a constant $c$ for all sufficiently large $q$.
\end{proof}

The remainder of this section is the proof of Theorem \ref{t2}.

\subsection{Some notation and basic facts}

Write $\1$ to represent $1^{d}$.

We use repeatedly the facts that $0$ is a prefix of every $B_{n}$ (Lemma \ref{0}) and that $B_{n}$ is a suffix of $B_{m}$ for $m \geq n$ for sufficiently large $n$ (due to $s_{n,r_{n}} = 0$).

We also use repeatedly the fact that $B_{n}B_{n}$ and $B_{n}\1 B_{n}$ are subwords of $B_{n+1}$ due to $s_{n,i}$ not being constant over $0 \leq i < r_{n}$.

\begin{lemma}\label{split}
There exists a constant $c \geq 1$ such that for all $n \geq N$, the words $B_{n}^{2}$ and $B_{n}\1 B_{n}$ differ on suffixes of length at least $h_{n} + c$.
\end{lemma}
\begin{proof}
Choose $c$ such that $B_{N}$ has $01^{c-1}$ as a suffix (possible as $B_{N}$ has $0$ as a prefix).  Since $B_{n}B_{n}$ has $B_{N}B_{n}$ as a suffix, $B_{n}B_{n}$ has $01^{c-1}B_{n}$ as a suffix.  As $B_{n}\1 B_{n}$ has $1^{c-1}1^{d}B_{n}$ as a suffix, this shows the words differ on the suffixes $01^{c-1}B_{n}$ and $1^{c}B_{n}$.
\end{proof}

\subsection{Counting via right-special words}

\begin{lemma}\label{8}
$B_{n} \in \mathcal{L}^{RS}$ for all $n$.
\end{lemma}
\begin{proof}
$B_{n+1}$ contains $B_{n}B_{n}$ and $B_{n}\1 $ as subwords.  $B_{n}$ has $0$ as a prefix so $B_{n} \in \mathcal{L}^{RS}$.
\end{proof}

\begin{lemma}\label{countmethod}
Write $f_{n} = p(h_{n}) - h_{n}$

If there are $t_{n}$ distinct right-special words, all of length at least $h_{n}$ and less than $q_{n}$, which are not suffixes of $B_{n+m}$ for any $m \geq 1$ then
\[
p(q_{n}) \geq q_{n} + f_{n} + t_{n}
\]
\end{lemma}
\begin{proof}
Since $p(q_{n}) - p(h_{n}) = |\{ w \in \mathcal{L}^{RS} : h_{n} \leq \len(w) < q_{n} \}|$, and since by Lemma \ref{8} we have at least $q_{n}-h_{n}$ suffixes of some $B_{n+m}$ of length at least $h_{n}$ and less than $q_{n}$ which are right-special and distinct from the $t_{n}$ hypothesized,
$
p(q_{n}) \geq p(h_{n}) + q_{n} - h_{n} + t_{n}
$.
\end{proof}

The proof of Theorem \ref{t2} will proceed by establishing the existence of right-special words which are not suffixes of any $B_{n+m}$.  To this end,
rewrite the defining words as
\[
B_{n+1} = (\prod_{j=1}^{z_{n}-1}B_{n}^{a_{n,j}}\1)B_{n}^{a_{n,z_{n}}}
\]
where $a_{n,j} \geq 1$ and $z_{n} \geq 2$ and $a_{n,j} \geq 2$ for at least one $j$ as $0$ and $d$ both occur in $\{ s_{n,i} : 0 \leq i < r_{n} \}$.

\subsection{The (straightforwardly) 5/3 cases}

Throughout this section, let $N$ such that $s_{n,i} \in \{ 0, d \}$ and $s_{n,r_{n}} = 0$ and $s_{n,i}$ is not constant over $0 \leq i < r_{n}$ for all $n \geq N$.

\begin{proposition}\label{15}
If for $n \geq N$ one of the following holds
\begin{itemize}
\item $a_{n,z_{n}} \geq 2$ and $a_{n,1} = 1$, i.e.~$B_{n}\1\dblank B_{n}^{2}$
\item $a_{n,z_{n}} = 1$ and $a_{n,1} \geq 2$, i.e.~$B_{n}^{2}\dblank\1 B_{n}$
\item $a_{n,z_{n}} = 1$ and $a_{n,1} = 1$ and $a_{n,j} \geq 3$ for some $j$, i.e.~$B_{n}\1\dblank B_{n}^{3} \dblank\1 B_{n}$
\end{itemize}
then \ls[\frac{5}{3}]{c}.
\end{proposition}

\begin{lemma}\label{9} Words of the form $B_{n}\1 B_{n}\1\dblank$

For $n \geq N$, if $a_{n,1} = a_{n,2} = 1$ then $B_{n}\1 B_{n}\1 B_{n} \in \mathcal{L}^{RS}$.
\end{lemma}
\begin{proof}
Let $j$ minimal such that $a_{n,j} \geq 2$.

If $j > 3$ then $B_{n+1}$ has the subword $B_{n}^{a_{n,j-3}}\1 B_{n}^{a_{n,j-2}}\1 B_{n}^{a_{n,j-1}} \1 B_{n}^{a_{n,j}} = B_{n}\1 B_{n}\1 B_{n}\1 B_{n}^{a_{n,j}}$ which has $B_{n}\1 B_{n}\1 B_{n}\1 B_{n}^{2}$ as a prefix.

If $j = 3$ then the word $B_{n+1}\1 B_{n+1}$ has the subword $B_{n}^{a_{n,z_{n}}}\1 B_{n}^{a_{n,1}} \1 B_{n}^{a_{n,2}}\1 B_{n}^{a_{n,3}}$ as a subword which has $B_{n}\1 B_{n}\1 B_{n}\1 B_{n}^{2}$ as a subword.

Then $B_{n}\1 B_{n}\1 B_{n} \in \mathcal{L}^{RS}$ as $B_{n}\1 B_{n}\1 B_{n}\1$ and $B_{n}\1 B_{n}\1 B_{n}B_{n}$ are both subwords of $B_{n}\1 B_{n}\1 B_{n}\1 B_{n}^{2}$.
\end{proof}

\begin{lemma}\label{10} Words of the form $B_{n}\1 B_{n}^{2}\dblank B_{n}^{2}$

For $n \geq N$, if $a_{n,z_{n}} \geq 2$ and $a_{n,1} = 1$ and $a_{n,2} \geq 2$ then $B_{n}^{2}\1 B_{n} \in \mathcal{L}^{RS}$.
\end{lemma}
\begin{proof}
$B_{n+1}\1 B_{n+1} \in \mathcal{L}$ implies $B_{n}^{a_{n,z_{n}}}\1 B_{n}^{a_{n,1}}\1 B_{n}^{a_{n,2}} \in \mathcal{L}$ so $B_{n}^{2}\1 B_{n}\1  \in \mathcal{L}$ as $a_{n,z_{n}} \geq 2$ and $a_{n,1} = 1$.

$B_{n+1}B_{n+1} \in \mathcal{L}$ implies $B_{n}^{a_{n,z_{n}}}B_{n}^{a_{n,1}}\1 B_{n}^{a_{n,2}} \in \mathcal{L}$.  Since $a_{n,2} \geq 2$ this gives $B_{n}^{2}\1 B_{n}^{2}$ so $B_{n}^{2}\1 B_{n}0 \in \mathcal{L}$.
\end{proof}

\begin{lemma}\label{11} Words of the form $B_{n}^{2}\1\dblank$

For $n \geq N$, if $a_{n,1} = 2$ then $B_{n}\1 B_{n}^{2} \in \mathcal{L}^{RS}$.
\end{lemma}
\begin{proof}
$B_{n}^{a_{n,z_{n}-1}}\1 B_{n}^{a_{n,z_{n}}}B_{n}^{a_{n,1}} \in \mathcal{L}$ as it is a subword of $B_{n+1}B_{n+1}$ so, as $a_{n,z_{n}-1} \geq 1$ and $a_{n,z_{n}} + a_{n,1} \geq 3$, also $B_{n}\1 B_{n}^{3} \in \mathcal{L}$.
$B_{n}^{a_{n,z_{n}}}\1 B_{n}^{a_{n,1}}\1$ is a subword of $B_{n+1}\1 B_{n+1}$ so, as $a_{n,z_{n}}\geq 1$, also $B_{n}\1 B_{n}^{2}\1 \in \mathcal{L}$.
\end{proof}

\begin{lemma}\label{12} Words of the form $B_{n}^{3}\dblank$ or $\dblank B_{n}^{4} \dblank$

For $n \geq N$, if $a_{n,1} > 2$ or $a_{n,j} > 3$ for some $j$ then $B_{n}^{3} \in \mathcal{L}^{RS}$.
\end{lemma}
\begin{proof}
If $a_{n,j} \geq 4$, since $B_{n}^{a_{n,j}}\1$ is a subword of $B_{n+1}$ so is $B_{n}^{4}\1$.  If $a_{n,1} \geq 3$ then since $B_{n}^{a_{n,z_{n}}}B_{n}^{a_{n,1}}\1$ is a subword of $B_{n+1}B_{n+1}$ and $a_{n,z_{n}}+a_{n,1} \geq 4$, also $B_{n}^{4}\1 \in \mathcal{L}$.
\end{proof}

\begin{lemma}\label{14} Words of the form $B_{n}\1 \dblank \1 B_{n}^{3}\1 \dblank \1 B_{n}$

For $n \geq N$, if $a_{n,1} = a_{n,z_{n}} = 1$ and $a_{n,j} = 3$ for some $j > 1$ then $B_{n}\1 B_{n}^{2} \in \mathcal{L}^{RS}$.
\end{lemma}
\begin{proof}
The word $B_{n+1}B_{n+1} \in \mathcal{L}$ so $B_{n}^{a_{n,z_{n}-1}}\1 B_{n}^{a_{n,z_{n}}}B_{n}^{a_{n,1}}\1 \in \mathcal{L}$ so $B_{n}\1 B_{n}^{2}\1 \in \mathcal{L}$.  As $B_{n}^{a_{n,j-1}}\1 B_{n}^{a_{n,j}} \in \mathcal{L}$, also $B_{n}\1 B_{n}^{3} \in \mathcal{L}$.
\end{proof}

\begin{proof}[Proof of Proposition \ref{15}]
First consider when $a_{n,z_{n}} \geq 2$ and $a_{n,1} = 1$.  If $a_{n,2} = 1$ then Lemma \ref{9} gives $D_{n} = B_{n}\1 B_{n}\1 B_{n} \in \mathcal{L}^{RS}$.  Since $B_{n+1}$ has $B_{n}^{2}$ as a suffix, every suffix of $D_{n}$ of length at least $h_{n}+ c$ is not a suffix of $B_{n+1}$ (Lemma \ref{split}) and is right-special.
If $a_{n,2} \geq 2$ then Lemma \ref{10} gives $D_{n} = B_{n}^{2}\1 B_{n} \in \mathcal{L}^{RS}$ which likewise has the property that every suffix of $D_{n}$ of length at least $h_{n} + c$ is right-special and not a suffix of $B_{n+1}$.

Now consider when $a_{n,z_{n}} = 1$ and $a_{n,1} \geq 2$.  If $a_{n,1} = 2$ then Lemma \ref{11} gives $D_{n} = B_{n}\1 B_{n}^{2} \in \mathcal{L}^{RS}$.  As $\1 B_{n}$ is a suffix of $B_{n+1}$ in this case, again every suffix of $D_{n}$ of length at least $h_{n}+c$ is not a suffix of $B_{n+1}$ and is right-special.  If $a_{n,1} > 2$ then Lemma \ref{12} gives $D_{n} = B_{n}^{3}$ which has the same property.

Last consider the case when $a_{n,z_{n}} = 1$ and $a_{n,1} = 1$ and $a_{n,j} \geq 3$ for some $j$.  If $a_{n,j} = 3$ then Lemma \ref{14} gives $D_{n} = B_{n}\1 B_{n}^{2} \in \mathcal{L}^{RS}$ and as $\1 B_{n}$ is a suffix of $B_{n+1}$ in this case, $D_{n}$ has the same property as above.  If $a_{n,j} > 3$ then Lemma \ref{12} gives $D_{n} = B_{n}^{3}$ with the same property.

In all cases, we have a word $D_{n}$ of length at least $3h_{n}$ with every suffix of length at least $h_{n}+c$ being right-special and not a suffix of $B_{n+1}$, so $2h_{n} - c$ right-special words which are not suffixes of $B_{n+1}$ all of length less than $3h_{n}$.  By Lemma \ref{countmethod}, then
$p(3h_{n}) \geq 3h_{n} + f_{n} + 2h_{n} - c = \frac{5}{3}(3h_{n}) + f_{n} - c$.
\end{proof}

\subsection{Words of the form \texorpdfstring{$B_{n}^{2} \dblank B_{n}^{2}$}{BB---BB}}

\begin{proposition}\label{18a}
If $a_{n,1} \geq 2$ and $a_{n,z_{n}} \geq 2$ and $a_{n+1,z_{n+1}} \geq 2$ then \ls{3c}.
\end{proposition}

These subshifts include the examples studied in \cite{ferenczichacon} defined by $B_{n+1} = B_{n}^{p}1B_{n}^{q}$ for $p,q > 1$.

The proof of Proposition \ref{18a} is a series of lemmas.
Write
\[
B_{n+1} = B_{n}^{\alpha}\1 uB_{n}^{\beta}
\]
for some word $u$ which is either empty or ends in $\1$.  Then $\alpha,\beta \geq 2$.

\begin{lemma}\label{18a1}
If $\alpha \ne \beta$ then \ls{c}.
\end{lemma}
\begin{proof}
The word $B_{n+1}B_{n+1}$ has $\1 B_{n}^{\alpha+\beta}\1$ as a subword so $B_{n}^{\alpha+\beta-1} \in \mathcal{L}^{RS}$.  The word $B_{n+1}$ has $\1 B_{n}^{\beta}$ as a suffix so our word differs from $B_{n+1}$ on suffixes of length at least $\beta h_{n} + c$ (Lemma \ref{split}, which we henceforth use implicitly) and so gives at least $(\alpha-1)h_{n} - c$ right-special words which are not suffixes of $B_{n+1}$ with length less than $(\alpha+\beta-1)h_{n}$.
Then by Lemma \ref{countmethod},
\[
p((\alpha+\beta-1)h_{n}) \geq (\alpha+\beta-1)h_{n} + f_{n} + (\alpha-1)h_{n} - c
= \frac{3}{2}(\alpha+\beta-1)h_{n} + \frac{1}{2}(\alpha-1-\beta)h_{n} + f_{n} - c
\]

If $\alpha \geq \beta + 1$ then $\frac{3}{2}(\alpha+\beta-1)h_{n} + \frac{1}{2}(\alpha-1+\beta)h_{n} + f_{n} - c \geq \frac{3}{2}(\alpha+\beta-1)h_{n} + f_{n} - c$.

Now consider when $\alpha < \beta$.  Let $\alpha^{\prime}$ minimal such that $\1 B_{n}^{\alpha^{\prime}} \1$ is a subword of $\1 B_{n+1}$.  Then $\alpha^{\prime} \leq \alpha < \beta$.  If $\alpha^{\prime} < \alpha$ then $B_{n}^{\alpha^{\prime}}\1 B_{n}^{\alpha^{\prime}}\1$ is a subword of $B_{n+1}$ as $\alpha^{\prime}$ is minimal so $B_{n}^{\alpha^{\prime}}$ must precede $\1 B_{n}^{\alpha^{\prime}} \1$ in $B_{n+1}$.  If $\alpha^{\prime} = \alpha$ then, as $\alpha < \beta$, the word $B_{n}^{\alpha^{\prime}}\1 B_{n}^{\alpha^{\prime}}\1$ is a subword of $B_{n+1}\1 B_{n+1}$ (with the first $\1$ in our word being the middle $\1$ in $B_{n+1}\1 B_{n+1}$).
Since $\alpha^{\prime}$ is minimal, $B_{n+1}$ has $B_{n}^{\alpha^{\prime}}\1 B_{n}^{\beta}$ as a suffix and, as $\alpha^{\prime} < \beta$, that word has $B_{n}^{\alpha^{\prime}}\1 B_{n}^{\alpha^{\prime}}B_{n}$ as a subword.
Then $B_{n}^{\alpha^{\prime}}\1 B_{n}^{\alpha^{\prime}} \in \mathcal{L}^{RS}$.

Since $B_{n+1}$ has $B_{n}^{\alpha^{\prime}+1}$ as a suffix, our word gives at least $\alpha^{\prime} h_{n} + d - c$ right-special words which are not suffixes of $B_{n+1}$ with length less than $2\alpha^{\prime} h_{n}+d $.  Then by Lemma \ref{countmethod} (which we will henceforth use implicitly),
\[
p(2\alpha^{\prime} h_{n}+d) \geq 2\alpha^{\prime} h_{n} + d + f_{n} + \alpha^{\prime} h_{n} + d - c = \frac{3}{2}(2\alpha^{\prime} h_{n} + d) + \frac{1}{2}d - c + f_{n}\qedhere
\]
\end{proof}

From here on, assume $\alpha = \beta$.

\begin{lemma}\label{18a2}
If $\1 B_{n}^{t} \1$ is a subword of $B_{n+1}$ for some $t \ne \beta$ and $t \ne 2\beta$ then \ls{c}.
\end{lemma}
\begin{proof}
As $B_{n+1}$ has $B_{n}^{\beta}\1$ as a prefix, there is some $t^{\prime} \ne \beta, 2\beta$ such that $B_{n}^{\beta}\1 B_{n}^{t^{\prime}}\1$ is a subword of $B_{n+1}$.

Suppose first that there is such a $t^{\prime} < \beta$.  As $B_{n}^{\beta}\1 B_{n}^{\beta}$ is a subword of $B_{n+1}\1 B_{n+1}$, then $B_{n}^{\beta}\1 B_{n}^{t^{\prime}} \in \mathcal{L}^{RS}$.  Since $B_{n}^{t^{\prime}+1}$ is a suffix of $B_{n+1}$ (as $t^{\prime} < \beta$), this gives at least $\beta h_{n} + d - c$ right-special suffixes that are not suffixes of $B_{n+1}$, all of length less than $(\beta+t^{\prime})h_{n} + d$.  Then as $t^{\prime} < \beta$,
\begin{align*}
p((\beta+t^{\prime})h_{n} + d) &\geq (\beta + t^{\prime})h_{n} + d + f_{n} + \beta h_{n} + d - c \\
&= \frac{3}{2}((\beta+t^{\prime})h_{n} + d) + \frac{1}{2}(\beta - t^{\prime})h_{n} + \frac{1}{2}d + f_{n} - c 
> \frac{3}{2}((\beta+t^{\prime})h_{n} + d) + f_{n} - c
\end{align*}

So we may assume that for all $t$ such that $\1 B_{n}^{t} \1$ is a subword of $B_{n+1}$, we have $t \geq \beta$.

Suppose now that $\beta < t^{\prime} < 2\beta$.  As $B_{n+1}$ has $B_{n}^{\beta}\1 B_{n}^{\beta}$ as a suffix (since we have ruled out $t < \beta$), the word $B_{n+1}B_{n+1}$ has $B_{n}^{\beta}\1 B_{n}^{2\beta}$ as a subword.  Then $B_{n}^{\beta}\1 B_{n}^{t^{\prime}} \in \mathcal{L}^{RS}$ as $t^{\prime} < 2\beta$.  This gives at least $(\beta + t^{\prime} - \beta)h_{n} + d - c$ right-special suffixes which are not suffixes of $B_{n+1}$ (which ends in $\1 B_{n}^{\beta}$) all of length less than $(\beta + t^{\prime})h_{n} + d$.  Then, as $t^{\prime} > \beta$,
\begin{align*}
p((\beta + t^{\prime})h_{n} + d) &\geq (\beta + t^{\prime})h_{n} + d + f_{n} + t^{\prime}h_{n} + d - c \\
&= \frac{3}{2}((\beta+t^{\prime})h_{n} + d) + \frac{1}{2}(t^{\prime}-\beta)h_{n} + \frac{1}{2}d - c + f_{n} 
> \frac{3}{2}((\beta+t^{\prime})h_{n} + d) + \frac{1}{2}d - c + f_{n}
\end{align*}

If $t^{\prime} > 2\beta$ then $B_{n}^{2\beta + 1}\1 \in \mathcal{L}$ so $B_{n}^{2\beta} \in \mathcal{L}^{RS}$ which gives at least $\beta h_{n} - c$ right-special suffixes which are not suffixes of $B_{n+1}$, all of length less than $2\beta h_{n}$.  Then 
\[
p(2\beta h_{n}) \geq 2\beta h_{n} + f_{n} + \beta h_{n} - c = \frac{3}{2}(2\beta h_{n}) + f_{n} - c \qedhere
\]
\end{proof}

We are left with the case when every $\1 B_{n}^{t} \1$ in $B_{n+1}$ has $t = \beta$ or $t = 2\beta$.

\begin{lemma}\label{18a3}
If $\1 B_{n}^{2\beta} \1$ is a subword of $B_{n+1}$ then \ls{2c}.
\end{lemma}
\begin{proof}
First observe that $B_{n}^{2\beta -1} \in \mathcal{L}^{RS}$ which gives at least $(\beta-1)h_{n} - c$ right-special suffixes of length less than $(2\beta - 1)h_{n}$ which are not suffixes of $B_{n+1}$.

Choose $x, y \geq 1$ so that $B_{n+1}$ has $(B_{n}^{\beta}\1)^{x}B_{n}^{2\beta}$ as a prefix and $B_{n}^{2\beta}(\1 B_{n}^{\beta})^{y}$ as a suffix.

Then $B_{n+1}\1 B_{n+1}$ has the subword $(B_{n}^{\beta}\1)^{x+y+1}B_{n}^{2\beta}$ which means that $(B_{n}^{\beta}\1)^{x+y}B_{n}^{\beta} \in \mathcal{L}^{RS}$.  This gives at least $x(\beta h_{n} + d) - c$ right-special suffixes which are not suffixes of $B_{n+1}$ of length less than $(x+y+1)\beta h_{n} + (x+y)d$.  As there is no overlap between these and the suffixes of $B_{n}^{2\beta-1}$, this gives a total of at least $((x+1)\beta - 1)h_{n} + xd - 2c$ right-special suffixes of length less than $(x+y+1)\beta h_{n} + (x+y)d$ which are not suffixes of $B_{n+1}$.  Then
\begin{align*}
p((x+y+1)\beta &h_{n} + (x+y)d) \geq (x+y+1)\beta h_{n}+(x+y)d + f_{n} + ((x+1)\beta-1)h_{n} + xd - 2c \\
&= \frac{3}{2}((x+y+1)\beta h_{n}+ (x+y)d) + \frac{1}{2}(x+1-y)\beta h_{n} - h_{n} + \frac{1}{2}(x-y)d - 2c + f_{n}
\end{align*}
and, as $\beta \geq 2$, this means that if $x \geq y$ then
\begin{align*}
p((x+y+1)\beta h_{n} + (x+y)d) &\geq \frac{3}{2}((x+y+1)\beta h_{n}+(x+y)d) + \frac{1}{2}\beta h_{n} - h_{n} - 2c + f_{n} \\
&\geq \frac{3}{2}((x+y+1)\beta h_{n}+(x+y)d) - 2c + f_{n}
\end{align*}

So we may assume from here on that $x < y$.

Write $B_{n+1} = \big{(}\prod_{i=1}^{s} (B_{n}^{\beta}\1 )^{x_{i}} B_{n}^{2\beta}\1 \big{)}(B_{n}^{\beta}\1 )^{y-1}B_{n}^{\beta}$ for some $s \geq 1$ and $x_{i} \geq 1$ with $x_{1} = x$.  Choose $i^{\prime}$ such that $x_{i^{\prime}}$ is minimal and $i^{\prime}$ is the minimal such $i$.

First we consider the case when $i^{\prime} > 1$.  Then $x_{i^{\prime}} < x$ since otherwise we would have chosen $i^{\prime} = 1$.  Since $B_{n+1}\1$ has $B_{n}^{2\beta}\1 (B_{n}^{\beta}\1 )^{x_{s}} B_{n}^{2\beta}\1 (B_{n}^{\beta}\1 )^{y-1}B_{n}^{\beta}\1$ as a suffix, it also has $(B_{n}^{\beta}\1 )^{x_{i^{\prime}}+1} B_{n}^{2\beta} (\1 B_{n}^{\beta})^{y}\1$ as a suffix since $x_{s} \geq x_{i^{\prime}}$.  As $y > x_{i^{\prime}}$, then $(B_{n}^{\beta}\1 )^{x_{i^{\prime}}+1} B_{n}^{2\beta} (\1 B_{n}^{\beta})^{x_{i^{\prime}}+1} \1 \in \mathcal{L}$.

Since $i^{\prime} > 1$, $B_{n+1}$ has $(B_{n}^{\beta}\1 )^{x_{i^{\prime}-1}} B_{n}^{2\beta} (\1 B_{n}^{\beta})^{x_{i^{\prime}}} \1 B_{n}^{2\beta}$ as a subword.  Then $(B_{n}^{\beta}\1 )^{x_{i^{\prime}}+1} B_{n}^{2\beta} (\1 B_{n}^{\beta})^{x_{i^{\prime}}}\1 B_{n}^{2\beta} \in \mathcal{L}$ as $x_{i^{\prime}-1} \geq x_{i^{\prime}}+1$ by the choice of $i^{\prime}$, so
$(B_{n}^{\beta}\1 )^{x_{i^{\prime}}+1} B_{n}^{2\beta} (\1 B_{n}^{\beta})^{x_{i^{\prime}}}\1 B_{n}^{\beta}$ is right-special.

Since $x_{i^{\prime}} < x < y$ implies $x_{i^{\prime}} < y - 1$ and $B_{n+1}$ has $\1 (B_{n}^{\beta}\1)^{y-1}B_{n}^{\beta}$ as a suffix, this gives at least $(x_{i^{\prime}}+2)\beta h_{n} + (x_{i^{\prime}}+1) d - c$ right-special words which are not suffixes of $B_{n+1}$, all of length less than $(2x_{i^{\prime}}+4)\beta h_{n}+ (2x_{i^{\prime}}+2)d$.  Therefore 
\begin{align*}
p((2x_{i^{\prime}}+4)\beta h_{n} + (2x_{i^{\prime}}+2)d) &\geq (2x_{i^{\prime}}+4)\beta h_{n} + (2x_{i^{\prime}}+2)d + f_{n} + (x_{i^{\prime}}+2)\beta h_{n} + (x_{i^{\prime}}+1) d - c \\
&= \frac{3}{2}((2x_{i^{\prime}}+4)\beta h_{n} + (2x_{i^{\prime}}+2)d) - c + f_{n}
\end{align*}

Now consider when $i=1$, i.e.~$x_{i} \geq x$ for all $i$.  Here $B_{n+1}B_{n+1}$ has $B_{n}^{2\beta}\1 (B_{n}^{\beta}\1 )^{y-1} B_{n}^{2\beta}\1 (B_{n}^{\beta}\1 )^{x-1} B_{n}^{2\beta}$ as a subword and $B_{n+1}\1$ has $(B_{n}^{\beta}\1 )^{x_{s}} B_{n}^{2\beta}\1 (B_{n}^{\beta}\1 )^{y-1}B_{n}^{\beta}\1$ as a subword.  As $x < y$ and $x \leq x_{s}$, this means $(B_{n}^{\beta}\1 )^{x}B_{n}^{2\beta}\1 (B_{n}^{\beta}\1 )^{x-1}B_{n}^{\beta} \in \mathcal{L}^{RS}$.

This gives at least $(x+1)\beta h_{n} + x d - c$ right-special words which are not suffixes of $B_{n+1}$, all of length less than $(2x+2)\beta h_{n}+2xd$.  Therefore 
\[
p((2x+2)\beta h_{n}+2xd) \geq (2x+2)\beta h_{n}+ 2xd + f_{n} + (x+1)\beta h_{n} + x d - c
= \frac{3}{2} ((2x+2)\beta h_{n} + 2xd) + f_{n} - c \qedhere
\]
\end{proof}

\subsubsection{Proof of Proposition \ref{18a}}

\begin{proof}[Proof of Proposition \ref{18a}]
By Lemmas \ref{18a1}, \ref{18a2} and \ref{18a3}, we are left with the situation when $B_{n+1} = (B_{n}^{\beta}\1)^{L}B_{n}^{\beta}$ for some $L \geq 1$.

Since $B_{n+2}$ has $B_{n+1}B_{n+1}$ as a suffix, and since Proposition \ref{15} then covers the case when $B_{n+2}$ has $B_{n+1}\1$ as a prefix, we may assume that $B_{n+2} = B_{n+1}^{\alpha_{n+1}}\1uB_{n+1}^{\beta_{n+1}}$ for some $\alpha_{n+1},\beta_{n+1} \geq 2$ where $u$ is either empty or ends with $\1$.  Lemmas \ref{18a1} and \ref{18a2} applied to $n+1$ mean we may assume $\alpha_{n+1} = \beta_{n+1}$.

As $B_{n+2}B_{n+2} \in \mathcal{L}$, the word $B_{n+1}^{2\beta_{n+1}}\1 \in \mathcal{L}$.  Then $B_{n+1}^{2\beta_{n+1}-1} \in \mathcal{L}^{RS}$.  As $B_{n+2}$ has $\1 B_{n}^{\beta_{n+1}}$ as a suffix, this gives at least $(\beta_{n+1}-1)h_{n+1} - c$ right-special words of length less than $(2\beta_{n+1}-1)h_{n+1}$ which are not suffixes of $B_{n+2}$.

As $B_{n+2}\1 B_{n+2} \in \mathcal{L}$ and $B_{n+2}$ has $B_{n+1}B_{n+1}$ as a prefix, $B_{n+2}\1 B_{n+2}$ has $B_{n+1}\1 B_{n+1} B_{n}^{\beta}$ as a subword.  Then $B_{n+1}\1 B_{n+1}B_{n}^{\beta} = (B_{n}^{\beta}\1)^{2L+1}B_{n}^{2\beta} \in \mathcal{L}$.  Therefore $(B_{n}^{\beta}\1)^{2L}B_{n}^{\beta} \in \mathcal{L}^{RS}$.  As $B_{n+2}$ has $B_{n+1}B_{n+1}$ as a suffix and that word has $B_{n}^{2\beta}(\1 B_{n}^{\beta})^{L}$ as a suffix, this gives at least $(L\beta -1) h_{n} + Ld - c$ right-special words of length less than $(2L+1)(\beta h_{n} + d)$ which are not suffixes of $B_{n+2}$.

As $B_{n}^{2\beta}\1$ is a subword of $B_{n+1}B_{n+1}$, this means $B_{n}^{2\beta-1} \in \mathcal{L}^{RS}$ which gives at least $(\beta - 1)h_{n} - c$ right-special words of length less than $(2\beta - 1)h_{n}$ which are not suffixes of $B_{n+1}$ hence not of $B_{n+2}$ as $B_{n+1}$ has $\1 B_{n}^{\beta}$ as a suffix.

As none of these right-special words overlap with one another, the three cases above provide at least $(\beta_{n+1}-1)h_{n+1} + (L\beta -1) h_{n} + (\beta - 1)h_{n} + Ld - 3c$ right-special words which are not suffixes of $B_{n+2}$ all of length less than $(2\beta_{n+1}-1)h_{n+1}$.

Since $h_{n+1} = (L+1)\beta h_{n} + Ld$, we then have $(\beta_{n+1}-1)h_{n+1} + (L\beta -1 + \beta - 1)h_{n} + Ld -3c = \beta_{n+1} h_{n+1} - 2h_{n} - 3c  = \beta_{n+1}h_{n+1} - \frac{2}{(L+1)\beta}(h_{n+1}-Ld) - 3c$ extra right-special words of length at most $(2\beta_{n+1}-1)h_{n+1}$.  Therefore, since $L \geq 1$ and $\beta \geq 2$ so $\frac{2}{(L+1)\beta} \leq \frac{2}{4}$, we have
\begin{align*}
p((2\beta_{n+1}-1)h_{n+1}) &\geq (2\beta_{n+1}-1)h_{n+1} + f_{n} + \beta_{n+1}h_{n+1} - \frac{2}{(L+1)\beta}h_{n+1} - 3c \\
&= \frac{3}{2}(2\beta_{n+1}-1)h_{n+1} + f_{n} + \Big{(}\frac{1}{2} - \frac{2}{(L+1)\beta}\Big{)}h_{n+1} - 3c \\
&\geq \frac{3}{2}(2\beta_{n+1}-1)h_{n+1} + f_{n} - 3c \qedhere
\end{align*}
\end{proof}

\subsection{Words of the form \texorpdfstring{$B_{n}\1  \dblank \1 B_{n}^{2}\1 \dblank \1 B_{n}$}{B1---1BB1---1B} with \texorpdfstring{$B_{n}^{3}$}{BBB} never appearing}

This section handles the most difficult case when $a_{n,1} = a_{n,z_{n}} = 1$ and $a_{n,j} \leq 2$ for all $j$.  This difficulty is likely unavoidable as this case contains the examples we exhibit which are near $1.5q$ in complexity.

\begin{proposition}\label{17}
If for infinitely many $n \geq N$, it holds that $a_{n,1} = a_{n,z_{n}} = a_{n+1,z_{n+1}} = 1$ and $a_{n,j} \leq 2$ for all $j$ and $a_{n,j} = 2$ for at least one $j$ then for all sufficiently large $n$ \ls{2c}.
\end{proposition}

Let $n \geq N$ such that $a_{n,1} = a_{n,z_{n}} = a_{n,z_{n+1}} = 1$ and $a_{n,j} \leq 2$ for all $j$ and $a_{n,j} = 2$ for at least one $j$.  Then we may write
\[
B_{n+1} = (B_{n}\1 )^{\alpha}(B_{n}^{2}\1)^{\beta}u(B_{n}^{2}\1)^{\kappa}(B_{n}\1 )^{\gamma-1}B_{n}
\]
for some word $u$, which has prefix $B_{n}\1$ and suffix $\1 B_{n} \1$,
and where $\alpha,\beta,\gamma,\kappa \geq 1$, or else $u$ is empty and $\kappa = 0$ and $\alpha, \beta, \gamma \geq 1$.  Then
\begin{align*}
B_{n+1}B_{n+1} &=~\dblank B_{n}^{2}\1 (B_{n}\1 )^{\gamma-1}B_{n}^{2}\1(B_{n}\1 )^{\alpha-1}(B_{n}^{2}\1)^{\beta}B_{n}\1 \dblank \\
B_{n+1}\1 B_{n+1} &=~\dblank B_{n}^{2}\1 (B_{n}\1 )^{\gamma+\alpha}(B_{n}^{2}\1)^{\beta}B_{n}\1 \dblank
\end{align*}

The proof of Proposition \ref{17} is a series of lemmas.

\begin{lemma}\label{alphaextra}
There are at least $\alpha (h_{n}+d)-c$ right-special words which are not suffixes of $B_{n+1}$ and with length less than $(\alpha + \gamma + 1)(h_{n}+d)$, all of which do not contain $B_{n}^{2}$ as a subword.
\end{lemma}
\begin{proof}
$B_{n}\1(B_{n}\1 )^{\gamma+\alpha}B_{n}^{2}$ is a subword of $B_{n+1}\1 B_{n+1}$ so $(B_{n}\1)^{\gamma+\alpha}B_{n} \in \mathcal{L}^{RS}$.  Since $B_{n+1}$ has suffix $B_{n}^{2}\1(B_{n}\1 )^{\gamma-1}B_{n}$, every suffix of $(B_{n}\1 )^{\gamma+\alpha}B_{n}$ at least $c$ longer than $(B_{n}\1 )^{\gamma}B_{n}$ is not a suffix of $B_{n+1}$.
\end{proof}

\begin{lemma}\label{4.14}
If $\alpha = 1$ and $\gamma = 1$ then \ls{2c}.
\end{lemma}
\begin{proof}
First consider the case when $\kappa = 0$ and $u$ is empty.  Here $B_{n+1} = (B_{n}\1)^{\alpha}(B_{n}^{2}\1)^{\beta}(B_{n}\1)^{\gamma-1}B_{n} = B_{n}\1 (B_{n}^{2}\1)^{\beta} B_{n}$.  Since $a_{n+1,z_{n+1}} = 1$ and $a_{n+1,j} \geq 2$ for some $j$, we have $B_{n+1}B_{n+1}\1 \in \mathcal{L}$.
Since $B_{n+1}B_{n+1}\1 = B_{n}\1 (B_{n}^{2}\1 )^{2\beta+1}B_{n}\1$, we then have $B_{n}\1 (B_{n}^{2}\1 )^{2\beta}B_{n} \in \mathcal{L}^{RS}$.

Since $B_{n+2}$ has $B_{n+1}\1 B_{n+1}$ as a suffix, it has $\1 B_{n} \1 B_{n} \1 (B_{n}^{2}\1 )^{\beta} B_{n}$ as a suffix.  This means our right-special word gives at least $(4\beta+2)h_{n} + (2\beta+1)d - (2\beta+2)h_{n} - (\beta+1)d - c = 2\beta h_{n} + \beta d - c$ right-special words which are not suffixes of $B_{n+2}$, all of length less than $(2\beta+1)(2h_{n}+d)$.  As Lemma \ref{alphaextra} gives at least $h_{n}+d-c$ additional right-special words which are not suffixes of $B_{n+2}$ and do not contains $B_{n}^{2}$, we conclude that 
\begin{align*}
p((2\beta+1)(2h_{n}+d)) &\geq (2\beta+1)(2h_{n}+d) + f_{n} + (2\beta + 1)h_{n}+(\beta+1)d - 2c \\
&= \frac{3}{2}(2\beta+1)(2h_{n}+d) + \frac{1}{2}d + f_{n} - 2c
\end{align*}

We now consider when $\kappa \geq 1$ and $u$ is nonempty.

Here $B_{n+1}B_{n+1} =~\dblank B_{n}\1 (B_{n}^{2}\1)^{\kappa+1+\beta}B_{n}\1 \dblank$ meaning that $B_{n}\1 (B_{n}^{2}\1)^{\beta+\kappa}B_{n} \in \mathcal{L}^{RS}$.  As $B_{n+1}$ has suffix $\1 B_{n}\1 (B_{n}^{2}\1)^{\kappa}B_{n}$, every suffix of our word of length at least $(2\kappa+2)h_{n}+(\kappa+1)d + c$ is not a suffix of $B_{n+1}$.  So there are at least $2\beta h_{n} + \beta d - c$ right-special words of length less than $2(\kappa + \beta + 1)h_{n}+(\kappa+\beta+1)d$ which are not suffixes of $B_{n+1}$.

 Lemma \ref{alphaextra} in this case also gives $h_{n}+d-c$ right-special words of length less than $3(h_{n}+d)$ which are not suffixes of $B_{n+1}$ and do not contain $B_{n}^{2}$.  So,
 \begin{align*}
 p(2(\kappa+\beta+1)h_{n}&+(\beta+\kappa+1)d) \geq 2(\kappa+\beta+1)h_{n}+(\beta+\kappa+1)d + f_{n} + (2\beta+1)h_{n} + (\beta + 1)d - 2c \\
& = \frac{3}{2}(2(\kappa+\beta+1)h_{n}+(\beta+\kappa+1)d) + f_{n} + (\beta - \kappa)h_{n} + \frac{1}{2}(\beta - \kappa + 1)d  - 2c
 \end{align*}
 so if $\beta \geq \kappa$ then
 \[
 p(2(\kappa+\beta+1)h_{n}+(\beta+\kappa+1)d) \geq \frac{3}{2}(2(\kappa+\beta+1)h_{n}+(\beta+\kappa+1)d) + f_{n} - 2c
 \]
 So from here on, assume $\beta < \kappa$.
 
Observe that if $(B_{n}\1 )^{4}B_{n} \in \mathcal{L}$ then necessarily $(B_{n}\1 )^{4}B_{n}^{2} \in \mathcal{L}$ as $\gamma = 1$, so $(B_{n}\1 )^{3}B_{n} \in \mathcal{L}^{RS}$.  As $B_{n+1}$ has $B^{2}\1 B_{n}$ as a suffix, every suffix of our word of length at least $2h_{n}+d+c$ is not a suffix of $B_{n+1}$.  This gives at least $2h_{n}+2d-c$ right-special words of length less than $4h_{n}+3d$ which are not suffixes of $B_{n+1}$.  Then 
\[
p(4h_{n}+3d) \geq 4h_{n}+3d + f_{n} + 2h_{n} + 2d - c = \frac{3}{2}(4h_{n}+3d) + f_{n} + \frac{1}{2}d - c
\]
So, from here on we assume that $\1 B_{n}\1 B_{n}\1 B_{n}\1  \notin \mathcal{L}$. 

Suppose that $B_{n}^{2}\1 B_{n}\1 B_{n}^{2}$ is a subword of $B_{n+1}$.  Then $B_{n}\1 B_{n}^{2} \1 B_{n} \1 B_{n} 0 \in \mathcal{L}$ as the initial $B_{n}^{2} \1$ is preceded by $B_{n} \1$.  Also, $B_{n+1}\1 B_{n+1}$ has the subword $B_{n}\1 B_{n}^{2}\1 B_{n} \1 B_{n} \1$ where the next-to-last $\1$ is the $\1$ appearing between the $B_{n+1}$ in $B_{n+1}\1 B_{n+1}$.  Then $B_{n}\1 B_{n}^{2}\1 B_{n}\1 B_{n} \in \mathcal{L}^{RS}$.

As $B_{n+1}$ has $B_{n}^{2}\1 B_{n}$ as a suffix, our word gives at least $3h_{n}+2d - c$ right-special words which are not suffixes of $B_{n+1}$, all of length less than $5h_{n}+3d$.  Therefore 
\[
p(5h_{n}+3d) \geq 5h_{n}+3d + f_{n} + 3h_{n}+2d - c = \frac{8}{5}(5h_{n}+3d) + \frac{1}{5}d + f_{n} - c
\]

So, from here on assume also that $B_{n}^{2} \1 B_{n} \1 B_{n}^{2}$ is not a subword of $B_{n+1}$.  Therefore we can write
\[
B_{n+1} = B_{n}\1 (\prod_{i=1}^{y} (B_{n}^{2}\1 )^{\beta_{i}} (B_{n}\1 )^{2}) (B_{n}^{2}\1 )^{\kappa} B_{n}
\]
for some $y \geq 1$ (since $u$ is nonempty) and $\beta_{i} \geq 1$ and $\beta_{1} = \beta < \kappa$.

Suppose first that $\beta_{y} < \beta$.   Set $m =  \min \{ \beta_{i} \}$.  Take $i$ such that $\beta_{i} = m$ then $\beta_{i} \leq \beta_{y} < \beta$ so $i > 1$.  Then $B_{n}\1(B_{n}^{2}\1)^{\beta_{i-1}}(B_{n}\1)^{2}(B_{n}^{2}\1)^{\beta_{i}}B_{n}\1 \in \mathcal{L}$.  This means $B_{n}\1 (B_{n}^{2}\1)^{m} (B_{n}\1)^{2} (B_{n}^{2}\1)^{m}B_{n}\1 \in \mathcal{L}$ as $m \leq \beta_{i-1}$.

$B_{n}\1 (B_{n}^{2}\1 )^{\beta_{y}} (B_{n}\1 )^{2} (B_{n}^{2}\1 )^{\kappa} B_{n}$ is a suffix of $B_{n+1}$, so $B_{n}\1 (B_{n}^{2}\1 )^{m} (B_{n}\1 )^{2} (B_{n}^{2}\1 )^{m} B_{n}^{2} \in \mathcal{L}$ since $m \leq \beta_{y}$ and $m < \beta < \kappa$.
Therefore $B_{n}\1 (B_{n}^{2}\1)^{m} (B_{n}\1)^{2} (B_{n}^{2}\1)^{m}B_{n} \in \mathcal{L}^{RS}$.

This gives at least $(2m+2)h_{n} + (m+2)d - c$ right-special words which are not suffixes of $B_{n+1}$, all of length less than $(4m+4)h_{n}+(2m+3)d$.  Therefore 
\begin{align*}
p((4m+4)h_{n}+(2m+3)d) &\geq (4m+4)h_{n}+(2m+3)d + f_{n} + (2m+2)h_{n} + (m+2)d - c \\
&= \frac{3}{2}((4m+4)h_{n}+(2m+3)d) + f_{n} + \frac{1}{2}d - c
\end{align*}

So, we may assume that $\beta_{y} \geq \beta$.

$B_{n+1}\1 B_{n+1}$ has the subword $(B_{n}^{2}\1)^{\kappa}B_{n}\1 B_{n}\1 (B_{n}^{2}\1)^{\beta} B_{n} \1$.  As $\beta < \kappa$, $B_{n}\1 (B_{n}^{2}\1 )^{\beta} B_{n}\1 B_{n}\1 (B_{n}^{2}\1 )^{\beta} B_{n}\1 \in \mathcal{L}$.

$B_{n+1}B_{n+1}$ has the subword
$B_{n}\1 (B_{n}^{2}\1 )^{\beta_{y}} (B_{n}\1 )^{2} (B_{n}^{2}\1 )^{\kappa} B_{n}B_{n}$ which has $B_{n}\1 (B_{n}^{2}\1 )^{\beta_{y}} (B_{n}\1 )^{2} (B_{n}^{2}\1 )^{\beta} B_{n}B_{n}$ as a subword.

As $\beta_{y} \geq \beta$, then $B_{n}\1 (B_{n}^{2}\1 )^{\beta} (B_{n}\1 )^{2} (B_{n}^{2}\1 )^{\beta} B_{n} \in \mathcal{L}^{RS}$.  Since $\beta < \kappa$, this gives at least $2(\beta+1)h_{n} + (\beta+2)d - c$ right-special words which are not suffixes of $B_{n+1}$, all of length less than $(4\beta+4)h_{n}+(2\beta+3)d$.  Therefore 
\begin{align*}
p((4\beta+4)h_{n}+(2\beta+3)d) &\geq (4\beta+4)h_{n}+(2\beta+3)d + f_{n} + 2(\beta+1)h_{n} + (\beta+2)d - c \\
&= \frac{3}{2}((4\beta+4)h_{n}+(2\beta+3)d) + f_{n} + \frac{1}{2}d -c \qedhere
\end{align*}
\end{proof}

\begin{lemma}\label{4.15}
If $\alpha = 1$ and $\gamma > 1$ then \ls{c}.
\end{lemma}
\begin{proof}
In this case, $B_{n+1}B_{n+1}$ contains the subword $B_{n}^{2}\1(B_{n}\1 )^{\gamma-1}(B_{n}^{2}\1)^{\beta+1}B_{n}$ and $B_{n+1}\1 B_{n+1}$ contains $B_{n}^{2}\1(B_{n}\1 )^{\gamma}(B_{n}^{2}\1)^{\beta}B_{n}\1 $.  Therefore $(B_{n}\1 )^{\gamma}(B_{n}^{2}\1)^{\beta}B_{n}^{2}$ appears in $B_{n+1}B_{n+1}$ and $(B_{n}\1 )^{\gamma}(B_{n}^{2}\1)^{\beta}B_{n}\1 $ in $B_{n+1}\1B_{n+1}$.  As $B_{n+1}$ has $\1 B_{n}\1 B_{n}$ as a suffix (as $\gamma > 1$), every suffix of $(B_{n}\1 )^{\gamma}(B_{n}^{2}\1)^{\beta}B_{n}$ longer than $01^{c-1}B_{n}\1 B_{n}$ is not a suffix of $B_{n+1}$ and is right-special.
This gives at least $(\gamma + 2\beta - 1)h_{n} + (\gamma + \beta - 1)d - c$ right-special words which are not suffixes of $B_{n+1}$ of length less than $(\gamma + 2\beta+1)h_{n} + (\gamma+\beta)d$.
Therefore, as $\gamma + 2\beta - 3 \geq 2 + 2 - 3 $,
\begin{align*}
p((\gamma+&2\beta+1)h_{n}+(\gamma + \beta)d) \\
&\geq (\gamma+2\beta+1)h_{n}+(\gamma + \beta)d + f_{n} + (\gamma + 2\beta - 1)h_{n} + (\gamma + \beta - 1)d - c \\
&= \frac{3}{2}((\gamma + 2\beta + 1)h_{n} + (\gamma + \beta)d) + \frac{1}{2}(\gamma + 2\beta - 3)h_{n} + \frac{1}{2}(\gamma + \beta - 2)d + f_{n} - c \\
&> \frac{3}{2}((\gamma + 2\beta + 1)h_{n} + (\gamma + \beta)d) + f_{n} - c \qedhere
\end{align*}
\end{proof}

\begin{lemma}\label{4.18}
If $\alpha > \gamma \geq 1$ then \ls{c}.
\end{lemma}
\begin{proof}
Lemma \ref{alphaextra} states there are at least $\alpha (h_{n}+d)-c$ right-special words which are not suffixes of $B_{n+1}$ all of length less than $(\alpha + \gamma + 1)(h_{n}+d)$.  Since $\alpha \geq \gamma + 1$,
\begin{align*}
p((\alpha+\gamma+1)(h_{n}+d)) &\geq (\alpha+\gamma+1)(h_{n}+d) + f_{n} + \alpha(h_{n}+d) - c \\
&= \frac{3}{2}(\alpha + \gamma + 1)(h_{n}+d) + \frac{1}{2}(\alpha - \gamma - 1)(h_{n}+d) + f_{n} - c \\
&\geq \frac{3}{2}(\alpha + \gamma + 1)(h_{n}+d) + f_{n} - c \qedhere
\end{align*}
\end{proof}

\begin{lemma}\label{4.17}
If $\alpha > 1$ and $\gamma > 1$ and $\beta > 1$ then \ls[2]{2c}.
\end{lemma}
\begin{proof}
  The word $(B_{n}\1 )^{\gamma}B_{n}^{2}\1 B_{n}\1 $ is a subword of $B_{n+1}B_{n+1}$ since $\alpha > 1$.  The word $(B_{n}\1 )^{\alpha+\gamma}B_{n}^{2}\1 B_{n}^{2}$ is a subword of $B_{n+1}\1 B_{n+1}$ since $\beta > 1$.  Therefore $(B_{n}\1 )^{\gamma}B_{n}^{2}\1 B_{n} \in \mathcal{L}^{RS}$.

Since $\gamma > 1$, $B_{n+1}$ has suffix $\1 B_{n}\1 B_{n}$ so there are at least $(\gamma + 1)h_{n} + \gamma d -c$ right-special words which are not suffixes of $B_{n+1}$ with length less than $(\gamma + 3)h_{n} + (\gamma+1)d$.  By Lemma \ref{alphaextra}, there are at least $\alpha (h_{n}+d)-c$ right-special words which are not suffixes of $B_{n+1}$ and with length less than $(\alpha + \gamma + 1)(h_{n}+d)$ and all with suffix $\1 B_{n}\1 B_{n}$ and which do not contain $B_{n}^{2}$ so there is no overlap with the right-special words already identified.

Therefore there are at least $(\gamma + 1 + \alpha)h_{n}+(\gamma+\alpha)d-2c$ right-special words which are not suffixes of $B_{n+1}$ all of length less than $(\gamma + 1 + \alpha)h_{n}+(\gamma+\alpha)d$ (as $\alpha \geq 2$ implies $\gamma+1+\alpha \geq \gamma + 3$).  Then 
\begin{align*}
p((\gamma+\alpha+1)h_{n}+(\gamma+\alpha)d)
&\geq (\gamma + \alpha + 1)h_{n} + (\gamma+\alpha)d + f_{n} + (\gamma + 1 + \alpha)h_{n} + (\gamma+\alpha)d - 2c \\
&= 2((\gamma + \alpha + 1)h_{n} + (\gamma+\alpha)d) + f_{n}- 2c \qedhere
\end{align*}
\end{proof}

\begin{lemma}\label{4.19}
If $\alpha > 1$ and $\gamma > 1$ and $B_{n}^{2}\1 B_{n}^{2}\1 \in \mathcal{L}$ then \ls{c}.
\end{lemma}
\begin{proof}
If the word $(B_{n}^{2}\1 )^{2} \in \mathcal{L}$ then necessarily $B_{n}\1 (B_{n}^{2}\1)^{2}B_{n}\1 \in \mathcal{L}$ since somewhere to the right of $(B_{n}^{2}\1)^{2}$ in $B_{n+1}$ must be $B_{n}\1$ as $\gamma > 1$.  Then $B_{n}\1 B_{n}^{2}\1 B_{n} \in \mathcal{L}^{RS}$ which gives at least $2h_{n}+d-c$ right-special words of length less than $4h_{n}+2d$ which are not suffixes of $B_{n+1}$.  Then 
\[
p(4h_{n}+2d) \geq 4h_{n}+2d + f_{n} + 2h_{n}+d -c = \frac{3}{2}(4h_{n}+2d) + f_{n} - c \qedhere
\]
\end{proof}

\subsubsection{The \texorpdfstring{$1 < \alpha \leq \gamma$}{1 <= alpha < gamma} and \texorpdfstring{$B_{n}^{2}\1 B_{n}^{2} \notin \mathcal{L}$}{BB1BB not appearing} case}

From here on, we assume $B_{n}^{2}\1 B_{n}^{2} \notin \mathcal{L}$.  Therefore we can write
\[
B_{n+1} = (\prod_{t=1}^{L}((B_{n}\1 )^{\alpha_{t}}B_{n}^{2}\1))(B_{n}\1 )^{\gamma-1}B_{n}
\]
for some $\alpha_{t} \geq 1$ and $L \geq 1$ where $\alpha_{1} = \alpha$ and we write $\alpha_{L+1} = \gamma - 1$.

\begin{lemma}\label{4.20}
If $1 < \alpha \leq \gamma$ and $\alpha_{t} < \gamma - 1$ for some $t \geq 2$ then \ls{c}.
\end{lemma}
\begin{proof}
Observe that $B_{n}\1 (B_{n}\1)^{\alpha_{k}} B_{n}^{2}\1 (B_{n}\1)^{\alpha_{k+1}}B_{n}^{2} \in \mathcal{L}$ for all $1 \leq k \leq L$ since, in the case when $k > 1$, it is a subword of $B_{n+1}$ and, in the case when $k=1$, it is a subword of $B_{n}\1 B_{n+1}$ which is a subword of $B_{n+1}\1 B_{n+1}$.

If $\alpha_{t+1} < \alpha_{k+1}$ for some $1 \leq t,k \leq L$ then the word $B_{n}\1(B_{n}\1 )^{\alpha_{k}}B_{n}^{2}\1(B_{n}\1 )^{\alpha_{k+1}}B_{n}^{2}$ has the subword $B_{n}\1(B_{n}\1 )^{\alpha_{k}}B_{n}^{2}\1(B_{n}\1 )^{\alpha_{t+1}}B_{n}\1 $.  As $B_{n}\1(B_{n}\1 )^{\alpha_{t}}B_{n}^{2}\1(B_{n}\1 )^{\alpha_{t+1}}B_{n}^{2} \in \mathcal{L}$, this implies that the word $B_{n}\1(B_{n}\1 )^{\min(\alpha_{t},\alpha_{k})}B_{n}^{2}\1(B_{n}\1 )^{\alpha_{t+1}}B_{n} \in \mathcal{L}^{RS}$.

Since $B_{n+1}$ ends in $B_{n}^{2}(\1 B_{n})^{\gamma}$, if $\alpha_{t+1} < \gamma - 1$ then suffixes of our right-special word which are longer than $01^{c-1}B_{n}\1 (B_{n}\1)^{\alpha_{t+1}}B_{n}$ are not suffixes of $B_{n+1}$.  This gives at least $(\min(\alpha_{t},\alpha_{k}) + 2)(h_{n}+d)-d-c$ right-special words which are not suffixes of $B_{n+1}$ which have length less than $(\min(\alpha_{t},\alpha_{k})+\alpha_{t+1}+4)(h_{n}+d)-2d$.

By hypothesis $\min \{ \alpha_{t} : t \geq 2 \} < \gamma - 1$.  Let $t$ such that $\alpha_{t+1} = \min \{ \alpha_{t} : t \geq 2 \}$.  Then there exists $k$ such that $\alpha_{t+1} < \alpha_{k+1}$ since the last $\alpha_{k}$ is followed by $\gamma - 1 > \alpha_{t+1}$.  Write $m = \min(a_{t},a_{k})$.  Then $a_{t+1} \leq m$ since it is chosen to be minimal.  
We then have, as $m - \alpha_{t+1} \geq 0$,
\begin{align*}
p((\alpha_{t+1}+m+4)(h_{n}+d)-2d) &\geq (\alpha_{t+1}+m+4)(h_{n}+d) -2d + f_{n} + (m+2)(h_{n}+d) - d - c \\
&= \frac{3}{2}((\alpha_{t+1}+m+4)(h_{n}+d)-2d) + \frac{1}{2}(m-\alpha_{t+1})(h_{n}+d) + f_{n} - c \\
&\geq \frac{3}{2}((\alpha_{t+1}+m+4)(h_{n}+d)-2d) + f_{n} - c \qedhere
\end{align*}
\end{proof}

\begin{lemma}\label{4.21}
If $1 < \alpha < \gamma$ and $\alpha_{t} \geq \gamma - 1$ for all $t \geq 2$ then \ls{c}.
\end{lemma}
\begin{proof}
The word $B_{n+1}B_{n+1}$ contains $B_{n}(\1 B_{n})^{\gamma-1}\1 B_{n}^{2}\1(B_{n}\1 )^{\alpha-1}B_{n}^{2}$ as a subword, so, as $\alpha < \gamma$, the word $(B_{n}\1 )^{\alpha}B_{n}^{2}\1(B_{n}\1 )^{\alpha-1}B_{n}0 \in \mathcal{L}$.

In the case $L = 1$, the word $B_{n+1}\1 = (B_{n}\1 )^{\alpha}B_{n}^{2}(\1 B_{n})^{\gamma}\1$ so $(B_{n}\1 )^{\alpha}B_{n}^{2}\1(B_{n}\1 )^{\alpha-1}B_{n}\1  \in \mathcal{L}$ since $\alpha < \gamma$.
In the case when $L > 1$, since $\alpha_{L} \geq \gamma - 1 \geq \alpha$, the word $B_{n+1}\1$ ends in $(B_{n}\1 )^{\alpha_{L}}B_{n}^{2}(\1 B_{n})^{\gamma}\1$ which has $(B_{n}\1 )^{\alpha}B_{n}^{2}\1(B_{n}\1 )^{\alpha-1}B_{n}\1 $ as a subword.

So $(B_{n}\1 )^{\alpha}B_{n}^{2}\1(B_{n}\1 )^{\alpha-1}B_{n} \in \mathcal{L}^{RS}$.  As $B_{n+1}$ has $(\1 B_{n})^{\gamma}$ as a suffix, our word gives at least $(\alpha+1)(h_{n}+d)-d-c$ right-special words which are not suffixes of $B_{n+1}$ all of length less than $(2\alpha+1)(h_{n}+d)-d$.  Then 
\begin{align*}
p((2\alpha+1)(h_{n}+d)-d) &\geq (2\alpha+1)(h_{n}+d)-d + f_{n} + (\alpha+1)(h_{n}+d) - d -c \\
&= \frac{3}{2}((2\alpha+1)(h_{n}+d)-d) + \frac{1}{2}h_{n} + f_{n} - c \qedhere
\end{align*}
\end{proof}

\begin{lemma}\label{4.22}
If $\alpha = \gamma > 1$ and $\alpha_{t} \geq \gamma - 1$ for all $t \geq 2$ and for some $t \geq 2$, $\alpha_{t} \geq \gamma$ with $\alpha_{t} \ne 2\gamma$ then \ls{2c}.
\end{lemma}
\begin{proof}
First consider the case when $\alpha_{t} > 2\gamma$.  As $B_{n}^{2}\1 (B_{n}\1)^{\alpha_{t}} B_{n}^{2}$ is a subword of $B_{n+1}$ and $\alpha_{t} \geq 2\gamma + 1$, this means $(B_{n}\1)^{2\gamma+2}B_{n}^{2} \in \mathcal{L}$.  Then $(B_{n}\1 )^{2\gamma+1}B_{n} \in \mathcal{L}^{RS}$.  Since $B_{n+1}$ has $B_{n}^{2}\1 (B_{n}\1 )^{\gamma-1}B_{n}$ as a suffix, there are at least $(\gamma+1)h_{n} + (\gamma +1)d - c$ right-special suffixes of our word all of length less than $(2\gamma+2)h_{n} + (2\gamma + 1)d$.  Then
\begin{align*}
p((2\gamma+2)h_{n} + (2\gamma + 1)d) &\geq (2\gamma+2)h_{n} + (2\gamma + 1)d + f_{n} + (\gamma + 1)h_{n} + (\gamma+1) d - c \\
&= \frac{3}{2}((2\gamma+2)h_{n} + (2\gamma + 1)d) +f_{n} + \frac{1}{2}d -c
\end{align*}

Next consider when $\alpha_{t} < 2\gamma$.  Then $B_{n}\1(B_{n}\1 )^{\gamma-1} B_{n}^{2}\1 (B_{n}\1 )^{\alpha_{t}}B_{n}^{2}$ is a subword of $B_{n+1}$ since $\alpha_{t-1} \geq \gamma - 1$.  As the word $(B_{n}\1 )^{\gamma} B_{n}^{2}\1 (B_{n}\1 )^{2\gamma}$ is a subword of $B_{n+1}\1 B_{n+1}$, this means $(B_{n}\1)^{\gamma}B_{n}^{2}\1 (B_{n}\1)^{\alpha_{t}}B_{n} \in \mathcal{L}^{RS}$.  Since $\alpha_{t} > \gamma - 1$ and $B_{n+1}$ has $B_{n}^{2}\1 (B_{n}\1)^{\gamma-1}B_{n}$ as a suffix, our word gives at least $(\gamma + 2 + \alpha_{t} + 1 - \gamma - 1)h_{n} + (\gamma + \alpha_{t} - \gamma)d - c$ right-special suffixes which are not suffixes of $B_{n+1}$, all of length less than $(\gamma + \alpha_{t} + 3)h_{n} + (\gamma+\alpha_{t}+1)d$.  Then, as $\alpha_{t} \geq \gamma$,
\begin{align*}
p((\gamma + \alpha_{t} + 3)h_{n} + &(\gamma+\alpha_{t}+1)d) \geq (\gamma + \alpha_{t} + 3)h_{n} + (\gamma+\alpha_{t}+1)d + f_{n} + (\alpha_{t}+2)h_{n} + (\alpha_{t}+1)d - 2c  \\
&= \frac{3}{2}((\gamma + \alpha_{t} + 3)h_{n} + (\gamma+\alpha_{t}+1)d) + \frac{1}{2}(\alpha_{t} - \gamma + 1)h_{n} + \frac{1}{2}(\alpha_{t} - \gamma + 1)d + f_{n} - 2c \\
&\geq \frac{3}{2}((\gamma + \alpha_{t} + 3)h_{n} + (\gamma+\alpha_{t}+1)d) + f_{n} - 2c \qedhere
\end{align*}
\end{proof}

\begin{lemma}\label{newthingsummary}
If $\alpha = \gamma > 1$ and $\alpha_{t} \in \{ \gamma-1, 2\gamma \}$ for all $t$ and $\alpha_{t} = 2\gamma$ for some $t$ then \ls{2c}.
\end{lemma}
\begin{proof}
Here we can write
\[
B_{n+1} = B_{n}\1 \Big{(} \prod_{i=1}^{s} ((B_{n}\1)^{\gamma-1}B_{n}^{2}\1)^{y_{i}} (B_{n}\1)^{2\gamma}B_{n}^{2}\1 \Big{)} ((B_{n}\1)^{\gamma-1}B_{n}^{2}\1)^{z} (B_{n}\1)^{\gamma-1}B_{n}
\]
for some $s \geq 1$ and $y_{i},z \geq 0$ and $y_{1} \geq 1$ (as $\alpha = \gamma > 1$).  Rearranging the grouping and writing $D_{n} = (B_{n}\1)^{\gamma}B_{n}$,
\begin{align*}
B_{n+1} &= B_{n}\1 \Big{(}\prod_{i=1}^{s} ((B_{n}\1)^{\gamma-1}B_{n}~B_{n}\1)^{y_{i}} (B_{n}\1)^{2\gamma}B_{n}~B_{n}\1\Big{)} ((B_{n}\1)^{\gamma-1}B_{n}~B_{n}\1)^{z} (B_{n}\1)^{\gamma-1}B_{n} \\
&= \Big{(}\prod_{i=1}^{s} ((B_{n}\1)^{\gamma}B_{n})^{y_{i}} (B_{n}\1)^{2\gamma+1}B_{n}\Big{)} ((B_{n}\1)^{\gamma}B_{n})^{z+1} \\
&= \Big{(}\prod_{i=1}^{s} D_{n}^{y_{i}}D_{n}\1 D_{n}\Big{)}D_{n}^{z+1}
= D_{n}^{y_{i}+1}\1 \Big{(}\prod_{i=2}^{s} D_{n}^{y_{i}+2}\1\Big{)} D_{n}^{z+2}
\end{align*}

Write $k_{n} = \len(D_{n})$.

First consider when $y_{1} > z$.  Since $B_{n+1}B_{n+1}$ has the subword $D_{n}^{z+2}D_{n}^{y_{1}+1}\1$ and $D_{n}$ has $0$ as a prefix (as $B_{n}$ does), then $D_{n}^{y_{1}+z+2} \in \mathcal{L}^{RS}$.  Since $D_{n}$ has $B_{n}$ as a suffix, this word disagrees with $B_{n+1}$ on suffixes longer than $1^{c-1}D_{n}^{z+2}$.  We then have at least $y_{1}k_{n} - c$ right-special words of length less than $(y_{1}+z+2)k_{n}$ which are not suffixes of $B_{n+1}$.

Lemma \ref{alphaextra} states there are at least $\gamma h_{n} + \gamma d = k_{n} - h_{n}$ right-special words of length less than $(2\gamma + 1)h_{n} + 2\gamma d$ which are not suffixes of $B_{n+1}$ and which do not contain $B_{n}^{2}$ as a subword, hence do not overlap with the words above.

Then, as $y_{1} \geq z + 1$ (and $k_{n} \geq 2h_{n}$ since $\gamma > 1$),
\begin{align*}
p((y_{1}+z+2)k_{n}) &\geq (y_{1}+z+2)k_{n} + f_{n} + y_{1}k_{n} - c + k_{n} - h_{n} - c \\
&= \frac{3}{2}(y_{1}+z+2)k_{n} + \frac{1}{2}(y_{1}-z)k_{n} - h_{n} + f_{n} - 2c \\
&\geq \frac{3}{2}(y_{1}+z+2)k_{n} + \frac{1}{2}k_{n} - h_{n} + f_{n} - 2c 
\geq \frac{3}{2}(y_{1}+z+2)k_{n} + f_{n} - 2c
\end{align*}

Now consider the case when $y_{i} < z$ for some $1 \leq i \leq s$.  Set $m = \min\{y_{i} : 1 \leq i \leq s\}$ so that $m < z$ and take $i$ minimal such that $y_{i}$ is minimal.

Since $B_{n+1}$ has $D_{n}^{y_{s}+2}\1 D_{n}^{z+2}$ as a suffix, then $D_{n}^{m+2}\1 D_{n}^{z+2} \in \mathcal{L}$.  When $i > 1$, as $D_{n}^{y_{i-1}+1}\1 D_{n}^{y_{i}+2} \1$ is a subword of $B_{n+1}$, then $D_{n}^{m+2}\1 D_{n}^{m+2}\1 \in \mathcal{L}$ as $y_{i-1} \geq y_{i} + 1$ as $i$ was taken minimal.  Then $D_{n}^{m+2}\1 D_{n}^{m+2} \in \mathcal{L}^{RS}$ as $m < z$.  As this word disagrees with suffixes of $B_{n+1}$ on words longer than $1^{c-1}D_{n}^{m+2}$, this gives at least $(m+2)k_{n} + d - c$ right-special words of length less than $2(m+2)k_{n} + d$ which are not suffixes of $B_{n+1}$.  Then
\[
p(2(m+2)k_{n}+d) \geq 2(m+2)k_{n} + d + f_{n} + (m+2)k_{n} + d - c
= \frac{3}{2}(2(m+2)k_{n}+d) + \frac{1}{2}d + f_{n} - c
\]

When $i = 1$, as $B_{n+1}$ has $D_{n}^{y_{s}+2}\1 D_{n}^{z+2}$ as a suffix (or $D_{n}^{y_{1}+1}\1 D_{n}^{z+2}$ in the case $s=1$), we have $D_{n}^{m+1}\1 D_{n}^{z+2} \in \mathcal{L}$.  The word $B_{n+1}\1 B_{n+1}$ has the subword $D_{n}^{z+2}\1 D_{n}^{y_{1}+1}\1$ which has $D_{n}^{m+1}\1 D_{n}^{m+1}\1$ as a subword.  As $m < z$, this means $D_{n}^{m+1}\1 D_{n}^{m+1} \in \mathcal{L}^{RS}$.  This word disagrees with suffixes of $B_{n+1}$ on words longer than $1^{c-1}D_{n}^{m+1}$ so there at least $(m+1)k_{n} + d - c$ right-special words of length less than $2(m+1)k_{n} + d$ which are not suffixes of $B_{n+1}$.  Then
\[
p(2(m+1)k_{n}+d) \geq 2(m+1)k_{n} + d + f_{n} + (m+1)k_{n} + d - c
= \frac{3}{2}(2(m+1)k_{n}+d) + \frac{1}{2}d + f_{n} - c
\]

From here on, assume that $y_{i} \geq z$ for all $i$.  We are left with the case when $y_{1} = z$.

Since $B_{n+1}\1 B_{n+1}$ has the subword $D_{n}^{z+2}\1 D_{n}^{y_{1}+1}\1 = D_{n}^{z+2}\1 D_{n}^{z+1}\1$ (as $y_{1} = z$) and $B_{n+1}$ has suffix $D_{n}^{y_{s}+2}\1 D_{n}^{z+2}$ which has the subword $D_{n}^{z+2}\1 D_{n}^{z+1} D_{n}$ (as $y_{s} \geq z$), this gives $D_{n}^{z+2}\1 D_{n}^{z+1} \in \mathcal{L}^{RS}$.  This word disagrees with suffixes of $B_{n+1}$ on words longer than $1^{c-1}D_{n}^{z+1}$ meaning there are at least $(z+2)k_{n} + d - c$ right-special words of length less than $(2z + 3)k_{n} + d$ which are not suffixes of $B_{n+1}$.  Then
\begin{align*}
p((2z + 3)k_{n} + d) &\geq (2z + 3)k_{n} + d + f_{n} + (z+2)k_{n} + d - c \\
&= \frac{3}{2}((2z + 3)k_{n} + d) + \frac{1}{2}k_{n} + \frac{1}{2}d + f_{n} - c
\geq \frac{3}{2}((2z + 3)k_{n} + d) + f_{n} - c \qedhere
\end{align*}
\end{proof}

\begin{lemma}\label{4.23}
If $\alpha = \gamma > 1$ and $\alpha_{t} = \gamma - 1$ for all $t \geq 2$ and $B_{n+2}$ has $B_{n+1}\1B_{n+1}$ as a suffix then \ls{2c}.
\end{lemma}
\begin{proof}
We are left with  $B_{n+1} = (B_{n}\1 )^{\gamma}B_{n}^{2}\1((B_{n}\1 )^{\gamma-1}B_{n}^{2}\1)^{L-1}(B_{n}\1 )^{\gamma-1}B_{n} = ((B_{n}\1)^{\gamma}B_{n})^{L+1}$.

Since $B_{n+1}B_{n+1}$ must occur somewhere in $B_{n+2}$ and not as a suffix, $B_{n+1}B_{n+1}\1 \in \mathcal{L}$, and since
$
B_{n+1}B_{n+1}\1 =((B_{n}\1)^{\gamma}B_{n})^{2L+2}\1
$
we have
$
((B_{n}\1)^{\gamma}B_{n})^{2L+1} \in \mathcal{L}^{RS}
$.
Since $h_{n+1} = (L+1)((\gamma+1)h_{n}+\gamma d)$, our right-special word has length
\begin{align*}
(2L+1)((\gamma+1)h_{n} + \gamma d) = 2h_{n+1} - ((\gamma+1)h_{n} + \gamma d)
\end{align*}
Since $B_{n+2}$ has $B_{n+1}\1 B_{n+1}$ as a suffix, this word disagrees with $B_{n+2}$ on suffixes of length at least $h_{n+1}+c$.  Therefore there are at least $h_{n+1} - ((\gamma+1)h_{n} + \gamma d) - c$ right-special suffixes of our word which are not suffixes of $B_{n+2}$.  

Lemma \ref{alphaextra} states there are also at least $\gamma (h_{n}+d) - c$ right-special words which do not have $B_{n}^{2}$ as a subword, hence do not overlap with those above nor with suffixes of $B_{n+2}$, all of length at most $(2\gamma+1)(h_{n}+d)$.  Then, as $\gamma > 1$,
\begin{align*}
p(2h_{n+1} &- ((\gamma+1)h_{n} + \gamma d)) \\
&\geq 2h_{n+1} - ((\gamma+1)h_{n} + \gamma d) + f_{n} + h_{n+1} - ((\gamma+1)h_{n} + \gamma d) - c + \gamma(h_{n}+d) - c \\
&= \frac{3}{2}(2h_{n+1} - ((\gamma+1)h_{n} + \gamma d)) + \frac{1}{2}(\gamma - 1)h_{n} + \frac{1}{2}\gamma d + f_{n} - 2c \\
&> \frac{3}{2}(2h_{n+1} - ((\gamma+1)h_{n} + \gamma d)) + f_{n} - 2c \qedhere
\end{align*}
\end{proof}

\subsubsection{Proof of Proposition \ref{17}}

\begin{proof}[Proof of Proposition \ref{17}]
Lemma \ref{4.14} gives $q_{n} \geq h_{n}$ such that $p(q_{n}) \geq 1.5q_{n} + f_{n} - 2c$ when $\alpha = \gamma = 1$.  Lemma \ref{4.15} takes care of $\alpha = 1$ and $\gamma > 1$. When $\alpha > \gamma \geq 1$, Lemma \ref{4.18} gives such a $q_{n}$.

We are left with the case when $\gamma \geq \alpha > 1$.  Lemma \ref{4.17} covers $\alpha, \gamma > 1$ and $\beta > 1$ so we proceed with $\beta = 1$.  Lemma \ref{4.19} covers the situation when $B_{n}^{2}\1 B_{n}^{2}\1 \in \mathcal{L}$ so we can assume that word does not appear from here on so $B_{n+1}$ is of the form written above Lemma \ref{4.20}.  That Lemma handles when $\alpha_{t} < \gamma - 1$ so we may assume $\alpha_{t} \geq \gamma - 1$ for all $t$.

Lemma \ref{4.21} then covers the case when $\alpha < \gamma$ so we may proceed with $\alpha = \gamma$.  Then Lemma \ref{4.22} shows that if $\alpha_{t} \geq \gamma$ with $\alpha_{t}\ne 2\gamma$ for some $t$ then we have such a $q_{n}$ so we may assume $\alpha_{t} \in \{ \gamma - 1, 2\gamma \}$ for all $t$.  Lemma \ref{newthingsummary} handles the case when $\alpha_{t} = 2\gamma$ for some $t$ so we can assume $\alpha_{t} = \gamma-1$ for all $t$.

By hypothesis, $a_{n+1,z_{n+1}} = 1$ meaning that $B_{n+2}$ ends with $B_{n+1}\1B_{n+1}$.  Lemma \ref{4.23} then guarantees the existence of such a $q_{n}$.

There are then $q_{n} \geq h_{n}$ with $p(q_{n}) \geq 1.5q_{n} + f_{n} - 2c$ for infinitely many, hence all sufficiently large $n$.
\end{proof}

\subsection{Proof of Theorem \ref{t2}}

\begin{proof}[Proof of Theorem \ref{t2}]
Set $C = 3c$.
Every $n \geq N$ satisfies one of (1) $a_{n,1} = 1$ and $a_{n,z_{n}} \geq 2$; (2) $a_{n,1} \geq 2$ and $a_{n,z_{n}} = 1$; (3) $a_{n,1} = a_{n,z_{n}} = 1$ and $a_{n,j} \geq 3$ for some $j$; (4) $a_{n,1},a_{n,z_{n}} \geq 2$; or (5) $a_{n,1} = a_{n,z_{n}} = 1$, $a_{n,j} \leq 2$ and $a_{n,j} = 2$ for some $j$.  At least one of those cases happens infinitely often.  For cases (1)--(3), Proposition \ref{15} gives the result.

For case (4), by Proposition \ref{trick3}, there exists a rank-one subshift generating the same language such that $a_{n,1} \geq 2$ and $a_{n,z_{n}} \geq 2$ and $a_{n+1,z_{n+1}} \geq 2$ for infinitely many $n$.  Proposition \ref{18a} applied to that subshift gives the claim.

If cases (1)--(4) all do not happen infinitely often then for all sufficiently large $n$, we are in case (5) in which case Proposition \ref{17} gives the claim.
\end{proof}

\section{Low complexity weakly mixing rank-one subshifts}\label{S4}

Consider the following class of rank-one subshifts:

\begin{definition}\label{theTs}
Let $L_{n} > 1$ and $\gamma_{n} > 1$ for all $n$.  Define the rank-one subshift with $B_{1} = 0$ and
\[
B_{n+1} = ((B_{n}1)^{\gamma_{n}}B_{n})^{L_{n}}
\]
\end{definition}

Observe that $h_{n+1} = L_{n}((\gamma_{n}+1)h_{n} + \gamma_{n})$ and
$
B_{n+1}B_{n+1} = ((B_{n}1)^{\gamma_{n}}B_{n})^{2L_{n}}
$ and
\[
B_{n+1}1B_{n+1} = ((B_{n}1)^{\gamma_{n}}B_{n})^{L_{n}-1}(B_{n}1)^{2\gamma_{n}+1}B_{n}((B_{n}1)^{\gamma_{n}}B_{n})^{L_{n}-1}
\]

\subsection{Right-special words}

\begin{lemma}\label{1Bn}
Let $w \in \mathcal{L}^{RS}$ with $1B_{n}$ as a suffix.  Then $w$ is a suffix of $(B_{n}1)^{2\gamma_{n}}B_{n}$ or $w$ is a suffix of $B_{n+1}$ or $w$ has $B_{n+1}$ as a proper suffix.
\end{lemma}
\begin{proof}
Observe that $1B_{n}$ is always preceded by $B_{n}$ so $w$ shares a suffix with $B_{n}1B_{n}$.

First consider when $w$ has $0B_{n}1B_{n}$ as a suffix.  As $B_{n}1$ is always preceded by $B_{n}$ or $1$, in this case $w$ shares a suffix with $B_{n}^{2}1B_{n}$.  Since $1B_{n}0$ only appears as a prefix of $1B_{n}^{2}$, having $w0 \in \mathcal{L}$ would then mean $B_{n}^{2}1B_{n}^{2} \in \mathcal{L}$ but that word is not in $\mathcal{L}$ since $\gamma_{n} > 1$.

So $w$ has $1B_{n}1B_{n}$ as a suffix (or else is a suffix of $B_{n}1B_{n}$ which is a suffix of $B_{n+1}$) and therefore shares a suffix with $B_{n}1B_{n}1B_{n}$.  Following the same logic, if $w$ shares a suffix with $0(B_{n}1)^{t}B_{n}$ then $w0$ shares a suffix with $0(B_{n}1)^{t}B_{n}0$ which can only occur as a subword of $B_{n}^{2}1(B_{n}1)^{t-1}B_{n}^{2}$, requiring that $t \geq \gamma_{n}$.

So $w$ shares a suffix with $B_{n}1(B_{n}1)^{\gamma_{n}-1}B_{n} = (B_{n}1)^{\gamma_{n}}B_{n}$.  As $B_{n}1$ is always preceded by $1$ or $B_{n}$, we have two cases to consider (if $w$ is a suffix of $(B_{n}1)^{\gamma_{n}}B_{n}$ then it is a suffix of $B_{n+1}$).

First consider when $w$ has $1(B_{n}1)^{\gamma_{n}}B_{n}$ as a suffix.  The only occurrence of that word is in $B_{n+1}1B_{n+1}$ and it is always preceded by $(B_{n}1)^{\gamma_{n}}B_{n}$ so $w$ must share a suffix with $(B_{n}1)^{2\gamma_{n}+1}B_{n}$.  Since $1(B_{n}1)^{2\gamma_{n}}B_{n}1 \notin \mathcal{L}$ as $(B_{n}1)^{2\gamma_{n}+1}$ is always preceded by $B_{n}$ (as $L_{n} > 1$) and since $w1 \in \mathcal{L}$, either $w$ is a suffix of $(B_{n}1)^{2\gamma_{n}}B_{n}$ or $w$ has $0(B_{n}1)^{2\gamma_{n}}B_{n}$ as a suffix.  Since $0(B_{n}1)^{2\gamma_{n}}B_{n}0 \notin \mathcal{L}$ because  $B_{n}(B_{n}1)^{2\gamma_{n}}B_{n}B_{n} = B_{n}^{2}\1 (B_{n}\1)^{2\gamma_{n}-1}B_{n}^{2} \notin \mathcal{L}$ as $\gamma_{n} > 1$, it must be that $w$ is a suffix of $(B_{n}1)^{2\gamma_{n}}B_{n}$.

Now consider when $w$ has $0(B_{n}1)^{\gamma_{n}}B_{n}$ as a suffix.  Then $w$ shares a suffix with $B_{n}(B_{n}1)^{\gamma_{n}}B_{n}$.  Since $(B_{n}1)^{\gamma_{n}}B_{n}$ is always preceded by $(B_{n}1)^{\gamma_{n}}B_{n}$ or $1$,
then $w$ shares a suffix with $((B_{n}1)^{\gamma}B_{n})^{2}$.  Then $w1$ shares a suffix with $((B_{n}1)^{\gamma}B_{n})^{2}1$ and since $((B_{n}1)^{\gamma}B_{n})^{2}1$ is always a suffix of $B_{n+1}1$, this shows that $w$ shares a suffix with $B_{n+1}$.
Then either $w$ is a suffix of $B_{n+1}$ or $w$ has $B_{n+1}$ as a proper suffix.
\end{proof}

\begin{lemma}\label{0BnRS}
Let $w \in \mathcal{L}^{RS}$ with $0B_{n}$ as a suffix.  Then $w$ is a suffix of $((B_{n-1}1)^{\gamma_{n-1}}B_{n-1})^{L_{n-1}-1}B_{n}$ and $n>1$.
\end{lemma}
\begin{proof}
Since $0B_{1} = 00$ and $000 = B_{1}^{3} \notin \mathcal{L}$, we have $n > 1$.

Every occurrence of $B_{n}$ appears either as $1B_{n}1$ or $1B_{n}B_{n}1$.  The word $0B_{n}$ is not a subword of $1B_{n}1$ and occurs as a subword of $1B_{n}B_{n}1$ at $L_{n-1}+1$ distinct starting locations.

The word $0B_{n}1$ only appears as a suffix of $1B_{n}B_{n}1$ since it must appear somewhere in $1B_{n}B_{n}1$ and the only appearance of $B_{n}1$ in that word is as a suffix as 
$B_{n}1 = ((B_{n-1}1)^{\gamma_{n-1}}B_{n-1})^{L_{n-1}-1} (B_{n-1}1)^{\gamma_{n-1}+1}$ and $(B_{n-1}1)^{\gamma_{n-1}+1}$ is not a subword of $B_{n}B_{n} = ((B_{n-1}1)^{\gamma_{n-1}}B_{n-1})^{2L_{n-1}}$.

So $w1$ shares a suffix with $B_{n}B_{n}1$ so $w$ shares a suffix with $B_{n}B_{n}$.  Since $B_{n}0$ must be a prefix of $B_{n}^{2}$ and $B_{n}^{3} \notin \mathcal{L}$, then
$B_{n}B_{n}0 \notin \mathcal{L}$.   As $w0 \in \mathcal{L}$, $w$ is then a proper suffix of $B_{n}B_{n}$.

Suppose $w$ has $0((B_{n-1}1)^{\gamma_{n-1}}B_{n-1})^{L_{n-1}-1}B_{n}$ as a suffix.  As that word only appears as a 
subword of $B_{n}B_{n}$ when the leading $0$ is the tail $0$ of the first $(B_{n-1}1)^{\gamma_{n-1}}B_{n-1}$ in the first $B_{n}$ of $B_{n}^{2}$, the word $0((B_{n-1}1)^{\gamma_{n-1}}B_{n-1})^{L_{n-1}-1}B_{n}0 \notin \mathcal{L}$ as $0((B_{n-1}1)^{\gamma_{n-1}}B_{n-1})^{L_{n-1}-1}B_{n}$ must be a suffix of $B_{n}B_{n}$ hence be followed by a $1$.  But then $w0 \notin \mathcal{L}$.

Suppose that $w$ has $1((B_{n-1}1)^{\gamma_{n-1}}B_{n-1})^{L_{n-1}-1}B_{n}$ as a suffix.  As $1 B_{n-1}$ is always preceded by $B_{n-1}$, then $w$ would share a suffix with $B_{n-1}1((B_{n-1}1)^{\gamma_{n-1}}B_{n-1})^{L_{n-1}-1}B_{n}$ but that contains $(B_{n-1}1)^{\gamma_{n-1}+1}$ as a subword which is not a subword of $B_{n}B_{n} = (B_{n-1}1)^{\gamma_{n-1}}B_{n-1})^{2L_{n-1}+}$.

Therefore $w$ must be a suffix of $((B_{n-1}1)^{\gamma_{n-1}}B_{n-1})^{L_{n-1}-1}B_{n}$.
\end{proof}

\begin{proposition}\label{words}
Let $w \in \mathcal{L}^{RS}$ with $\len(w) > 1$.  Then there exists a unique $n$ such that exactly one of the following holds (and for $m \ne n$, none of them hold):
\begin{itemize}
\item $w$ is a suffix of $B_{n+1}$ and $h_{n} < \len(w) \leq h_{n+1}$
\item $w$ is a suffix of $(B_{n}1)^{2\gamma_{n}}B_{n}$ and $(\gamma_{n}+1)h_{n} + \gamma_{n} < \len(w) \leq (2\gamma_{n}+1)h_{n} + 2\gamma_{n}$
\item $w$ is a suffix of $((B_{n-1}1)^{\gamma_{n-1}}B_{n-1})^{L_{n-1}-1}B_{n}$ and $h_{n} < \len(w) \leq h_{n}(2 - \frac{1}{L_{n-1}})$ and $n > 1$
\end{itemize}
In all three cases, $h_{n} < \len(w) \leq h_{n+1}$.
\end{proposition}
\begin{proof}
As $11 \notin \mathcal{L}$, $w$ must end in $0$.  Let $n$ be the largest integer such that $w$ has $B_{n}$ as a proper suffix (such $n$ exists since $B_{1} = 0$).  Then $w$ has either $0B_{n}$ or $1B_{n}$ as a suffix.

Lemma \ref{1Bn} states that if $w$ has $1B_{n}$ as a suffix then either $w$ is a suffix of $(B_{n}1)^{2\gamma_{n}}B_{n}$ or is a suffix of $B_{n+1}$, which are the second and first cases of the proposition, respectively, or else $w$ has $B_{n+1}$ as a proper suffix which would contradict the choice of $n$.

Lemma \ref{0BnRS} states that if $w$ has $0B_{n}$ as a suffix then $n > 1$ and $w$ is a suffix of $((B_{n-1}1)^{\gamma_{n-1}}B_{n})^{L_{n-1}-1}B_{n}$.  This puts us in the third case as $(\gamma_{n-1}+1)h_{n-1}+\gamma_{n-1} = \frac{1}{L_{n-1}}h_{n}$.

Suffixes of $(B_{n}1)^{2\gamma_{n}}B_{n}$ of length less than or equal to $(\gamma_{n}+1)h_{n} + \gamma_{n}$ are suffixes of $(B_{n}1)^{\gamma_{n}}B_{n}$ which is a suffix of $B_{n+1}$ but all suffixes longer than that are not suffixes of $B_{n+1}$ as $B_{n+1}$ has $0(B_{n}1)^{\gamma_{n}}B_{n}$ as a suffix.  Suffixes of $((B_{n-1}1)^{\gamma_{n-1}}B_{n-1})^{L_{n-1}-1}B_{n}$ of length at least $h_{n}+1$ have $0B_{n}$ as a suffix so are not suffixes of $B_{n+1}$ as $B_{n+1}$ has $1B_{n}$ as a suffix.  Clearly there is no overlap between the second and third cases as the second has $1B_{n}$ as a suffix and the third has $0B_{n}$ as a suffix.  Therefore the length restrictions make the cases a partition of $\mathcal{L}^{RS}$.

Since $(2 - \frac{1}{L_{n-1}})h_{n} < (2\gamma_{n}+1)h_{n} + 2\gamma_{n} < 2((\gamma_{n}+1)h_{n} + \gamma_{n}) \leq L_{n}((\gamma_{n}+1)h_{n}+\gamma_{n}) = h_{n+1}$, 
in all three cases $h_{n} < \len(w) \leq h_{n+1}$.
\end{proof}

\subsection{The complexity function}

\begin{proposition}\label{count}
The complexity function satisfies $p(h_{2}+1) = h_{2}(1 + \frac{1}{L_{1}}) + 1$ and for $q > h_{2}$, choosing $n$ to be the unique integer such that $h_{n} < q \leq h_{n+1}$,
\[
p(q+1) - p(q) = \left\{ \begin{array}{ll} 2 \quad & \text{when } h_{n} < q \leq (2 - \frac{1}{L_{n-1}})h_{n} \\
1 & \text{when }  (2 - \frac{1}{L_{n-1}})h_{n} < q \leq (\gamma_{n}+1)h_{n} + \gamma_{n} \\
2 & \text{when }  (\gamma_{n} + 1)h_{n} + \gamma_{n} < q \leq (2\gamma_{n}+1)h_{n} + 2\gamma_{n} \\
1 & \text{when }  (2\gamma_{n}+1)h_{n} + 2\gamma_{n} < q \leq h_{n+1}
\end{array} \right.
\]
\end{proposition}
\begin{proof}
In Proposition \ref{words}, there is no overlap among $n$ since $h_{n} < \len(w) \leq h_{n+1}$ for all three cases.

Recall that $p(q+1) - p(q) = |\{ w \in \mathcal{L}^{RS} : \len(w) = q \}|$.

Let $q$ and $n$ such that $h_{n} < q \leq h_{n+1}$.  There is exactly one suffix of $B_{n+1}$ of length $q$.  There is a suffix of the second form in Proposition \ref{words} of length $q$ precisely when $(\gamma_{n}+1)h_{n} + \gamma_{n} < q \leq (2\gamma_{n}+1)h_{n}+2\gamma_{n}$.
There is a suffix of the third form in Proposition \ref{words} of length $q$ precisely when $h_{n} < q \leq (2 - \frac{1}{L_{n-1}})h_{n}$ and $n > 1$.

For $1 < q \leq h_{2}$, Proposition \ref{words} applies with $n=1$ and the third case is vacuous.  Then
 $p(q+1)-p(q) = 1$ for $1 < q \leq (\gamma_{1}+1)h_{1} + \gamma_{1} = 2\gamma_{1} + 1$.  For $2\gamma_{1} + 1 < q \leq (2\gamma_{1}+1)h_{1} + 2\gamma_{1} = 4\gamma_{1} + 1$, we have $p(q+1) - p(q) = 2$ and for $4\gamma_{1}+1 < q \leq h_{2}$, $p(q+1) - p(q) = 1$.  Therefore, as $p(2) = 3$ and $h_{2} = L_{1}((\gamma_{1}+1)h_{1}+\gamma_{1}) = L_{1}(2\gamma_{1}+1)$,
\begin{align*}
p(h_{2}+1)  &= p(2) + (p(2\gamma_{1}+1) - p(2)) + (p(4\gamma_{1}+1) - p(2\gamma_{1}+1) + (p(h_{2}+1) - p(4\gamma_{1}+1)) \\
&= 3 + (2\gamma_{1}-1) + 2(2\gamma_{1} ) + (h_{2} - 4\gamma_{1})
= h_{2} + 2\gamma_{1} + 2
= h_{2} + \frac{1}{L_{1}}h_{2} + 1 \qedhere
\end{align*}
\end{proof}

\begin{theorem}\label{comp}
The transformations in Definition \ref{theTs} satisfy $p(h_{n+1}) = (1 + \frac{1}{L_{n}})h_{n+1}$.

If $\frac{\gamma_{n}}{h_{n}} \to 0$ then they also satisfy
\begin{align*}
\liminf \frac{p(q)}{q} &= 1 + \liminf \frac{1}{\max(L_{n-1},\gamma_{n}+1)} \\
\limsup \frac{p(q)}{q} &= \frac{3}{2} + \limsup \frac{1}{4\min(L_{n-1},\gamma_{n}+1)-2}
\end{align*}
\end{theorem}
\begin{proof}
For $n \geq 2$, by Proposition \ref{count},
\begin{align*}
p(\Big{(}2 - \frac{1}{L_{n-1}}\Big{)}h_{n}+1) - p(h_{n}+1) &= 2\Big{(}1 - \frac{1}{L_{n-1}}\Big{)}h_{n} \\
p((\gamma_{n}+1)h_{n} + \gamma_{n}+1)  -p(\Big{(}2 - \frac{1}{L_{n-1}}\Big{)}h_{n}+1) &= \Big{(}\gamma_{n} - 1 + \frac{1}{L_{n-1}}\Big{)}h_{n} + \gamma_{n} \\
p((2\gamma_{n}+1)h_{n} + 2\gamma_{n} + 1) - p((\gamma_{n}+1)h_{n} + \gamma_{n} + 1) &= 2(\gamma_{n}h_{n} + \gamma_{n}) \\
p(h_{n+1}+1) - p((2\gamma_{n}+1)h_{n} + 2\gamma_{n} + 1) &= h_{n+1} - (2\gamma_{n}+1)h_{n} - 2\gamma_{n}
\end{align*}
and therefore
\begin{align*}
p(h_{n+1}+1) - p(h_{n}+1) &= \Big{(}2 - \frac{2}{L_{n-1}} + \gamma_{n} - 1 + \frac{1}{L_{n-1}} + 2\gamma_{n} - 2\gamma_{n} - 1\Big{)}h_{n} + \gamma_{n} + 2\gamma_{n} - 2\gamma_{n} + h_{n+1} \\
&= h_{n+1} + \Big{(}\gamma_{n} - \frac{1}{L_{n-1}}\Big{)}h_{n} + \gamma_{n}
= h_{n+1} + (\gamma_{n}+1)h_{n} + \gamma_{n} - h_{n} - \frac{1}{L_{n-1}}h_{n} \\
&= h_{n+1} + \frac{1}{L_{n}}h_{n+1} - h_{n} - \frac{1}{L_{n-1}}h_{n}
\end{align*}
which implies that
\begin{align*}
p(h_{n+1}+1) &= p(h_{2} + 1) + \sum_{m=2}^{n} (p(h_{m+1} + 1) - p(h_{m}+1)) \\
&= 1 + \Big{(}1 + \frac{1}{L_{1}}\Big{)}h_{2} + \sum_{m=2}^{n} \Big{(}\Big{(}1 + \frac{1}{L_{m}}\Big{)}h_{m+1} - \Big{(}1 + \frac{1}{L_{m-1}}\Big{)}h_{m}\Big{)}
= 1 + \Big{(}1 + \frac{1}{L_{n}}\Big{)}h_{n+1}
\end{align*}
Since $p(h_{n+1}+1) - p(h_{n+1}) = 1$, then $p(h_{n+1}) = (1 + \frac{1}{L_{n}})h_{n+1}$.

Combining this with our initial observations,
\begin{align*}
p(\Big{(}2-\frac{1}{L_{n-1}}\Big{)}h_{n}+1) - 1 &= \Big{(}1 + \frac{1}{L_{n-1}}\Big{)}h_{n} + 2\Big{(}1 - \frac{1}{L_{n-1}}\Big{)}h_{n} = \Big{(}3 - \frac{1}{L_{n-1}}\Big{)}h_{n} \tag{$\dagger$}\\
p((\gamma_{n}+1)h_{n} + \gamma_{n}+1) -1 &= \Big{(}3 - \frac{1}{L_{n-1}}\Big{)}h_{n} + \Big{(}\gamma_{n}-1+\frac{1}{L_{n-1}}\Big{)}h_{n} + \gamma_{n}
= (\gamma_{n}+2)h_{n} + \gamma_{n} \\
p((2\gamma_{n}+1)h_{n} + 2\gamma_{n}+1) -1 &= (\gamma_{n}+2)h_{n} + \gamma_{n} + 2\gamma_{n}h_{n} + 2\gamma_{n}
= (3\gamma_{n}+2)h_{n} + 3\gamma_{n} \tag{$\ddagger$}
\end{align*}
and so
\begin{align*}
\frac{p(h_{n})}{h_{n}} &= 1 + \frac{1}{L_{n-1}} \\
\frac{p((2-\frac{1}{L_{n-1}})h_{n}+1)-1}{(2-\frac{1}{L_{n-1}})h_{n}} &= \frac{3 - \frac{1}{L_{n-1}}}{2 - \frac{1}{L_{n-1}}}
= \frac{3}{2} + \frac{\frac{1}{2}~\frac{1}{L_{n-1}}}{2 - \frac{1}{L_{n-1}}} = \frac{3}{2} + \frac{1}{4L_{n-1}-2} \\
\frac{p((\gamma_{n}+1)h_{n}+\gamma_{n}+1)-1}{(\gamma_{n}+1)h_{n}+\gamma_{n}} &= \frac{(\gamma_{n}+2)h_{n}+\gamma_{n}}{(\gamma_{n}+1)h_{n}+\gamma_{n}} = 1 + \frac{1}{\gamma_{n}+1+\frac{\gamma_{n}}{h_{n}}} \\
\frac{p((2\gamma_{n}+1)h_{n} + 2\gamma_{n}+1)-1}{(2\gamma_{n}+1)h_{n} + 2\gamma_{n}}
&= \frac{(3\gamma_{n}+2)h_{n} + 3\gamma_{n}}{(2\gamma_{n}+1)h_{n} + 2\gamma_{n}}
= \frac{3}{2} + \frac{\frac{1}{2}}{2\gamma_{n}+1+\frac{2\gamma_{n}}{h_{n}}}
\end{align*}
Now observe that, since $1 \leq p(q+1) - p(q) \leq 2$ for all $q$, the function $p(q)$ is increasing when $p(q+1) - p(q) = 2$ and decreasing when $p(q+1) - p(q) = 1$.  Therefore the $\liminf$ and $\limsup$ are attained along sequences of the four above-mentioned values.
Provided $\frac{\gamma_{n}}{h_{n}} \to 0$, then
\[
\liminf_{q} \frac{p(q)}{q} = \liminf_{n} \min\Big{(}1+\frac{1}{L_{n-1}}, 1 + \frac{1}{\gamma_{n}+1}\Big{)}
\]
and
\[
\limsup_{q} \frac{p(q)}{q} = \limsup_{n} \max\Big{(}\frac{3}{2} + \frac{1}{4L_{n-1}-2}, \frac{3}{2} + \frac{1}{4\gamma_{n}+2}\Big{)} \qedhere
\]
\end{proof}

\subsection{Complexity nearing 1.5q}

\begin{theorem}\label{minimal}
Let $\epsilon > 0$ and $f(q) \to \infty$.  Then there exists $\gamma_{n} = \gamma > 1$ and $L_{n} \to \infty$ such that the transformation in Definition \ref{theTs} satisfies
\[
\limsup \frac{p(q)}{q} < \frac{3}{2} + \epsilon
\quad\quad\text{and}\quad\quad
p(h_{n}) < h_{n} + f(h_{n})
\]
\end{theorem}
\begin{proof}
Choose $\gamma > 1$ such that $\frac{1}{4\gamma+2} < \epsilon$.

Given $h_{n}$, choose $q_{n}$ such that for all $q \geq q_{n}$, we have $f(q) > (\gamma+1)h_{n} + \gamma$.  Then choose $L_{n}$ such that $L_{n}((\gamma+1)h_{n}+\gamma) \geq q_{n}$.  Then by Theorem \ref{comp},
\[
p(h_{n+1}) = \Big{(}1 + \frac{1}{L_{n}}\Big{)}h_{n+1} = h_{n+1} + (\gamma+1)h_{n} + \gamma < h_{n+1} + f(q_{n}) \leq h_{n+1} + f(h_{n+1})
\]
Since $L_{n} \to \infty$, 
$
\limsup \frac{p(q)}{q} = \frac{3}{2} + \frac{1}{4\gamma+2} < \frac{3}{2} + \epsilon
$.
\end{proof}

\section{Weak mixing for rank-one transformations}\label{Swm}

\begin{theorem}\label{ismix}
Let $T$ be a rank-one transformation with bounded spacers (there exists $k$ such that $s_{n,i} \leq k$ for all $0\leq i < r_{n}$ and all $n$) and $\kappa > 0$ such that for all sufficiently large $n$,
\[
|\{ s_{n,i} = 0 : 0 \leq i < r_{n} \}| \geq \kappa (r_{n}+1) \quad\quad\text{and}\quad\quad |\{ s_{n,i} = 1 : 0 \leq i < r_{n} \}| \geq \kappa (r_{n}+1)
\]
Then $T$ is weakly mixing on a finite measure space.
\end{theorem}

We adapt the proof that Chacon's transformation is weakly mixing from \cite{silva2008invitation}.

\begin{lemma}[Lemma 2.7.3 \cite{silva2008invitation}]\label{273}
For any measurable set $A$ and $\epsilon > 0$, there exists $N$ such that for all $n \geq N$ there exists $Q \subseteq \{ 0, \ldots, h_{n} - 1 \}$ such that $\mu(A \symdiff \bigcup_{q \in Q} I_{n,q}) < \epsilon$.
\end{lemma}

\begin{lemma}[Lemma 3.7.3 \cite{silva2008invitation}]\label{373}
For any positive measure set $A$ and $\epsilon > 0$, there exists $N$ such that for all $n \geq N$ there exists $0 \leq a < h_{n}$ such that $\mu(A \cap I_{n,a}) > (1 - \epsilon)\mu(I_{n,a})$.
\end{lemma}

\begin{lemma}\label{Q}
Let $I$ a level and $A$ a measurable set such that $\mu(A \cap I) \geq \frac{3}{4}\mu(I)$.  For any $0 < \delta < 1$, there exists $N$ such that for all $n \geq N$, if $I = \bigsqcup_{q \in Q} I_{n,q}$ is the partition of $I$ into sublevels in $C_{n}$ then $|\{ q \in Q : \mu(A \cap I_{n,q}) \geq \delta \mu(I_{n,q}) \}| \geq \frac{1}{2}|Q|$.
\end{lemma}
\begin{proof}
Choose $\alpha > 0$ such that $\alpha < \frac{1}{4}(1 + \frac{1}{\delta})^{-1}$ so that
$\frac{\alpha}{\delta} + \alpha + \frac{1}{4} < \frac{1}{2}$.  Let $A_{1} = A \cap I$.  By Lemma \ref{273}, there exists $N$ such that for any $n \geq N$ there is $Q^{\prime} \subseteq Q$ such that if we set $I^{\prime} = \bigsqcup_{a\in Q^{\prime}}I_{n,a}$ then $\mu(A_{1} \symdiff I^{\prime}) < \alpha \mu(I)$.  Now observe that
$
\mu(I^{\prime} \symdiff I) \leq \mu(I^{\prime} \symdiff A_{1}) + \mu(A_{1} \symdiff I)
< \alpha \mu(I) + \frac{1}{4}\mu(I)
$.

Set $Q^{\prime\prime} = \{ a \in Q^{\prime} : \mu(I_{n,a} \setminus A_{1}) < \delta \mu(I_{n,a}) \}$ and $I^{\prime\prime} = \bigsqcup_{a\in Q^{\prime\prime}} I_{n,a}$.  Since $\delta \mu(I_{n,a}) \leq \mu(I_{n,a} \setminus A_{1})$ for $a \in Q^{\prime} \setminus Q^{\prime\prime}$,
\[
\delta \mu(I^{\prime} \symdiff I^{\prime\prime}) = \delta \mu(I^{\prime} \setminus I^{\prime\prime})
= \sum_{a \in Q^{\prime} \setminus Q^{\prime\prime}} \delta\mu(I_{n,a})
\leq \sum_{a \in Q^{\prime}\setminus Q^{\prime\prime}} \mu(I_{n,a} \setminus A_{1})
\leq \mu(I^{\prime} \setminus A_{1}) < \alpha \mu(I)
\]
so
$
\mu(I^{\prime\prime} \symdiff I) \leq \mu(I^{\prime\prime} \symdiff I^{\prime}) + \mu(I^{\prime} \symdiff I)
< \frac{\alpha}{\delta} \mu(I) + (\alpha + \frac{1}{4})\mu(I) < \frac{1}{2} \mu(I)
$.
Then
$\mu(I^{\prime\prime} \cap I) \geq \frac{1}{2}\mu(I)$ which means $|Q^{\prime\prime}| \geq \frac{1}{2}|Q|$.
\end{proof}

\begin{lemma}\label{wkmix}
If $T$ is on a finite measure space and there exists $\kappa > 0$ and $\{ t_{n,\ell} \}$ such that  for any two levels $I$ and $J$ in $C_{n}$, with $J$ being $\ell$ levels below $I$, $\mu(T^{t_{n,\ell}}I \cap I) \geq \kappa^{\ell} \mu(I)$ and $\mu(T^{t_{n,\ell}}I \cap J) \geq \kappa^{\ell} \mu(J)$ then $T$ is weakly mixing.
\end{lemma}
\begin{proof}
Let $A$ and $B$ be any positive measure sets.  By Lemma \ref{373}, there exist levels $I_{1}$ and $J_{1}$ in some column $C_{N}$ such that $\mu(A \cap I_{1}) > \frac{3}{4}\mu(I_{1})$ and $\mu(B \cap J_{1}) > \frac{3}{4}\mu(J_{1})$.
Let $0 \leq \ell < h_{N}$ such that $I_{1}$ is $\ell$ levels above $J_{1}$ (interchanging the roles of $A$ and $B$ if necessary).

Set $\delta = \frac{\kappa^{\ell}}{3}$.  By Lemma \ref{Q}, there exists $n > N$ such that if $I_{1} = \bigcup_{q\in Q_{1}} I_{n,q}$ and $J_{1} = \bigcup_{q \in Q_{2}} I_{n,q}$ then $|\{ q \in Q_{1} : \mu(A \cap I_{n,q}) \geq (1 - \delta) \mu(I_{n,q}) \}| \geq \frac{1}{2}|Q_{1}|$ and $|\{ q \in Q_{2} : \mu(B \cap I_{n,q}) \geq (1 - \delta) \mu(I_{n,q}) \}| \geq \frac{1}{2}|Q_{2}|$.  Since $I_{1}$ is $\ell$ levels above $J_{1}$, $q \in Q_{1}$ if and only if $q - \ell \in Q_{2}$ and $|Q_{1}| = |Q_{2}|$.  Therefore
\begin{align*}
|\{ q \in Q_{1} : \mu(A \cap I_{n,q}) < (1-\delta) \mu(I_{n,q}) \text{ or } \mu(B \cap I_{n,q-\ell}) < (1-\delta) \mu(I_{n,q}) \}| < \frac{1}{2}|Q_{1}| + \frac{1}{2}|Q_{2}| = |Q_{1}|
\end{align*}
meaning there exists $q \in Q_{1}$ such that $I = I_{n,q}$ and $J=I_{n,q-\ell}$ satisfy $\mu(A \cap I) \geq (1 - \delta) \mu(I)$ and $\mu(B \cap J) \geq (1 - \delta) \mu(J)$.

By hypothesis, $\mu(T^{t_{n,\ell}}I \cap I) \geq \kappa^{\ell}\mu(I) = 3\delta\mu(I)$ and $\mu(T^{t_{n,\ell}}I \cap J) \geq \kappa^{\ell}\mu(J) = 3\delta\mu(I)$.  Set $A_{1} = A \cap I$ and $B_{1} = B \cap J$ so that $\mu(I \setminus A_{1}) < \delta \mu(I)$ and $\mu(J \setminus B_{1}) < \delta \mu(I)$.  Then
\begin{align*}
\mu(T^{t_{n,\ell}}A_{1} \cap B_{1}) &\geq \mu(T^{t_{n,\ell}}I \cap J) - \mu(I \setminus A_{1}) - \mu(J \setminus B_{1}) \geq 3\delta\mu(I) - \delta\mu(I) - \delta\mu(I) = \delta\mu(I) > 0
\end{align*}
and similarly
\begin{align*}
\mu(T^{t_{n,\ell}}A_{1} \cap A_{1}) &\geq \mu(T^{t_{n,\ell}}I \cap I) - \mu(I \setminus A_{1}) - \mu(I \setminus A_{1})
\geq 3\delta\mu(I) - \delta\mu(I) - \delta\mu(I) = \delta\mu(I) > 0
\end{align*}

Hence for all positive measure sets $A$ and $B$ there exists $t$ such that $\mu(T^{t}A \cap A) \geq \mu(T^{t}A_{1} \cap A_{1}) > 0$ and $\mu(T^{t}A \cap B) > 0$, which is equivalent to weak mixing (\cite{furstenberg}).
\end{proof}

\begin{lemma}\label{m1}
Let $\kappa > 0$ and $n \in \mathbb{N}$ and set $t_{n,\ell} = \sum_{t=0}^{\ell-1} h_{n+t}$.  Assume
\[
|\{ s_{n,i} = 0 : 0 \leq i < r_{n} \}| \geq \kappa (r_{n}+1) \quad\quad\text{and}\quad\quad |\{ s_{n,i} = 1 : 0 \leq i < r_{n} \}| \geq \kappa (r_{n}+1)
\]
Let $I$ and $J$ be levels in $C_{n}$ with $J$ being $\ell$ levels below $I$.  Then 
\[
\mu(T^{t_{n,\ell}}I \cap I) \geq \kappa^{\ell}\mu(I)
\quad\quad
\text{and}
\quad\quad
\mu(T^{t_{n,\ell}}I \cap J) \geq \kappa^{\ell}\mu(J)
\]
\end{lemma}
\begin{proof}
Write $I = I_{n,a}$ for some $0 \leq a < h_{n}$.
As
$
T^{h_{n}}I_{n,a} \supset \bigsqcup_{i < r_{n} : s_{n,i} = 0} I_{n,a}^{[i+1]} 
$,
applying this twice,
\[
T^{h_{n+1}+h_{n}}I_{n,a} \supset \bigsqcup_{i_{0} : s_{n,i_{0}}=0} T^{h_{n+1}}I_{n,a}^{[i+1]}
\supset \bigsqcup_{i_{0} : s_{n,i_{0}} = 0} \quad \bigsqcup_{i_{1} : s_{n+1,i_{1}}=0} I_{n,a}^{[i_{0}+1][i_{1}+1]}
\]
where $I_{n,a}^{[i][j]}$ has the obvious meaning: it is the $j^{th}$ sublevel of the $i^{th}$ sublevel of $I_{n,a}$ meaning $I_{n,a}^{[i][j]}$ is a level in $C_{n+2}$.  Continuing this process:
\[
T^{\sum_{t=0}^{\ell-1}h_{n+t}}I_{n,a} \supset \bigsqcup_{i_{0} : s_{n,i_{0}}=0} \quad \bigsqcup_{i_{1} : s_{n+1,i_{1}}=0} \cdots \bigsqcup_{i_{\ell-1}:s_{n+\ell-1,i_{\ell-1}}=0} I_{n,a}^{[i_{0}+1][i_{1}+1]\cdots[i_{\ell-1}+1]}
\]
Therefore
\begin{align*}
\mu(T^{\sum_{t=0}^{\ell-1}h_{n+t}}I_{n,a} \cap I_{n,a}) &\geq \sum_{i_{0}:s_{n,i_{0}}=0} \cdots \sum_{i_{\ell-1}:s_{n+\ell-1,i_{\ell-1}}=0} \mu(I_{n,a}^{[i_{0}+1]\cdots [i_{\ell-1}+1]}) 
\geq \Big{(}\prod_{t=0}^{\ell-1} \kappa(r_{n+t}+1)\Big{)} \mu(I_{n+\ell,a}) \\
&= \Big{(}\prod_{t=0}^{\ell-1} \kappa(r_{n+t}+1)\Big{)}\Big{(}\prod_{t=0}^{\ell-1} \frac{1}{r_{n+t}+1}\Big{)}\mu(I_{n,a})
= \kappa^{\ell} \mu(I_{n,a})
\end{align*}
Similarly,
$
T^{h_{n}}I_{n,a} \supset \bigsqcup_{i < r_{n} : s_{n,i} = 1} I_{n,a-1}^{[i+1]} 
$
 so
\[
T^{\sum_{t=0}^{\ell-1}h_{n+t}}I_{n,a} \supset \bigsqcup_{i_{0} : s_{n,i_{0}}=1} \quad \bigsqcup_{i_{1} : s_{n,i_{1}}=1} \cdots \bigsqcup_{i_{\ell-1}:s_{n+\ell-i,i_{\ell-1}}=1} I_{n,a-\ell}^{[i_{0}+1][i_{1}+1]\cdots[i_{\ell-1}+1]}
\]
As $J = I_{n,a-\ell}$ since $J$ is $\ell$ levels below $I$ in $C_{n}$,
\begin{align*}
\mu(T^{\sum_{t=0}^{\ell-1}h_{n+t}}I_{n,a} \cap J) &\geq \sum_{i_{0}:s_{n,i_{0}}=1} \cdots \sum_{i_{\ell-1}:s_{n+\ell-1,i_{\ell-1}}=1} \mu(I_{n,a-\ell}^{[i_{0}+1]\cdots [i_{\ell-1}+1]}) \geq \kappa^{\ell} \mu(I_{n,a-\ell}) \qedhere
\end{align*}
\end{proof}

\begin{proposition}\label{finmeas}
Let $T$ be a rank-one transformation.  If there exists a constant $k$ such that $s_{n,i} \leq k$ for all $0 \leq i \leq r_{n}$ for all sufficiently large $n$ then $T$ is on a finite measure space.
\end{proposition}
\begin{proof}
Writing $S_{n}$ for the spacers added above the $n^{th}$ column $C_{n}$, we have $\mu(S_{n}) = \sum_{i=0}^{r_{n}} s_{n,i}\mu(I_{n+1}) \leq k(r_{n}+1)\mu(I_{n+1}) = k\mu(I_{n}) = \frac{k}{h_{n}}\mu(C_{n})$.  Since $h_{n} \geq \prod_{j=1}^{n-1} (r_{j} + 1) \geq 2^{n-1}$, then $\mu(C_{n+1}) \leq (1 + \frac{k}{2^{n-1}}) \mu(C_{n})$ so $\lim \mu(C_{n}) \leq \mu(C_{0})\prod_{n=0}^{\infty} (1 + \frac{k}{2^{n}}) < \infty$.
\end{proof}

\begin{proof}[Proof of Theorem \ref{ismix}]
Lemmas \ref{wkmix} and \ref{m1} and Proposition \ref{finmeas}.
\end{proof}

\subsection{Weak mixing for low complexity transformations}

\begin{corollary}\label{theTswkmix}
The subshifts in Definition \ref{theTs} are weakly mixing (on finite measure spaces) provided that $\limsup \gamma_{n} < \infty$.
\end{corollary}
\begin{proof}
Since $B_{n+1} = ((B_{n}1)^{\gamma_{n}}B_{n})^{L_{n}}$, we have $|\{ 0 \leq i < r_{n} : s_{n,i} = 0 \}| = L_{n} - 1$ and $|\{ 0 \leq i < r_{n} : s_{n,i} = 1 \}| = L_{n}\gamma_{n}$.  As $r_{n}+1 = L_{n}(\gamma_{n}+1)$, this means
\[
\frac{|\{ 0 \leq i < r_{n} : s_{n,i} = 0 \}|}{r_{n}+1} = \frac{L_{n}-1}{L_{n}(\gamma_{n}+1)} \geq \frac{1}{\gamma_{n}+1}~\frac{1}{2}
\]
Likewise $|\{ i : s_{n,i} = 1 \}|/(r_{n}+1) = \gamma_{n}/(\gamma_{n}+1)$.
As $\gamma_{n}$ is bounded, Theorem \ref{ismix} gives weak mixing.
\end{proof}

\begin{theorem}\label{Areal}
For every $\epsilon > 0$, there exists a weakly mixing rank-one transformation (on a probability space) such that the associated subshift has complexity $\limsup \nicefrac{p(q)}{q} < 1.5 + \epsilon$.

For any $f(q) \to \infty$, the subshifts can be made to satisfy $p(q) < q + f(q)$ infinitely often.
\end{theorem}
\begin{proof}
Corollary \ref{theTswkmix} and Theorem \ref{minimal}.
\end{proof}

\begin{theorem}\label{msjreal}
For every $\epsilon > 0$, there exists a subshift with complexity satisfying
$\limsup \nicefrac{p(q)}{q} < 1.5 + \epsilon$ and $\liminf \nicefrac{p(q)}{q} < 1 + \epsilon$
such that the associated rank-one transformation is weakly mixing (on a probability space) and has minimal self-joinings (hence also has trivial centralizer and is mildly mixing).
\end{theorem}
\begin{proof}
For $\epsilon > 0$, let $\gamma > 1$ such that $\frac{1}{\gamma+1} < \epsilon$.  Then the transformation in Definition \ref{theTs} with $\gamma_{n} = \gamma$ and $L_{n} = \gamma + 1$ satisfies, by Theorem \ref{comp},
\[
\liminf \frac{p(q)}{q} = 1 + \frac{1}{\gamma+1} < 1 + \epsilon \quad\quad\text{and}\quad\quad \limsup \frac{p(q)}{q} = \frac{3}{2} + \frac{1}{4\gamma-2} < \frac{3}{2} + \epsilon
\]
and Corollary \ref{theTswkmix} gives weak mixing.  As $\{ r_{n} \}$ is bounded, Ryzhikov's theorem \cite{ryzhikovmsj} gives minimal self-joinings (the transformations are non-rigid since the $s_{n,i}$ are not constant over $0 \leq i < r_{n}$ hence are not ``flat" in the sense of Theorem 2 in \cite{ryzhikovmsj}).
\end{proof}

\begin{remark}
The examples with $p(q) < q + f(q)$ such that $L_{n} \to \infty$ are most likely not mildly mixing, hence do not have minimal self-joinings.  In essence, any alternative construction of those examples (where $f(q)/q \to 0$ so $L_{n} \to \infty$) which has bounded spacers necessarily involves constructing $B_{n+1}^{\prime} = ((B_{n}1)^{\gamma}B_{n})^{\ell_{n}}$ with $\ell_{n}$ uniformly bounded followed by $B_{n+1} = (B_{n+1}^{\prime})^{L_{n}/\ell_{n}}$.  As the second step involves adding no spacers, the construction is ``flat" and therefore should admit a rigid factor.
\end{remark}

\subsection{Totally ergodic subshifts with \texorpdfstring{$\limsup \nicefrac{p(q)}{q} = 1.5$}{limsup p(q)/q = 1.5}}

\begin{theorem}\label{Dreal}
For any $f(q) \to \infty$, there exists a totally ergodic rank-one subshift (on a probability space) satisfying
$p(q) < 1.5q + f(q)$ for all sufficiently large $q$ and $p(h_{n}) < h_{n} + f(h_{n})$ for all $n \geq 2$.
\end{theorem}
\begin{proof}
Let $f^{*}(q) = \inf \{ f(q^{\prime}) : q^{\prime} \geq q \}$.  Then $f^{*}(q)$ is nondecreasing and $f^{*}(q) \to \infty$.

Set $\gamma_{1} = L_{1} = 2$.  Given $\gamma_{n-1}$ and $L_{n-1}$ (and therefore $h_{n}$), choose $\gamma_{n}$ such that $\frac{1}{2}h_{n} < f^{*}(\gamma_{n})$.  Then choose $L_{n} = m_{n}!$ for some $m_{n} > n$ such that $(\gamma_{n}+1)h_{n} + \gamma_{n} < f^{*}(L_{n})$.

As $h_{n+1} = L_{n}((\gamma_{n}+1)h_{n} + \gamma_{n})$, we then have $\frac{1}{L_{n}}h_{n+1} < f^{*}(L_{n}) \leq f^{*}(h_{n+1})$.
Theorem \ref{comp} gives that
\begin{align*}
p(h_{n}) &= \Big{(}1 + \frac{1}{L_{n-1}}\Big{)}h_{n} < h_{n} + f^{*}(h_{n}) \leq h_{n} + f(h_{n})
\end{align*}

The count $(\dagger)$ in the proof of Theorem \ref{comp} gives that
\begin{align*}
p(\Big{(}2-&\frac{1}{L_{n-1}}\Big{)}h_{n}+1) = \Big{(}3 - \frac{1}{L_{n-1}}\Big{)}h_{n} + 1 = \frac{3}{2}\Big{(}2 - \frac{1}{L_{n-1}}\Big{)}h_{n} + \frac{1}{2} \frac{1}{L_{n-1}} h_{n} + 1 \\
&\leq \frac{3}{2}\Big{(}2 - \frac{1}{L_{n-1}}\Big{)}h_{n} + \frac{1}{2}f^{*}(h_{n}) + 1
< \frac{3}{2}(\Big{(}2 - \frac{1}{L_{n-1}}\Big{)}h_{n} + 1) + f^{*}(\Big{(}2-\frac{1}{L_{n-1}}\Big{)}h_{n}+1)
\end{align*}
and the count $(\ddagger)$ in the proof of Theorem \ref{comp} gives
\begin{align*}
p((2\gamma_{n}+1)h_{n} + 2\gamma_{n}+1) &= (3\gamma_{n}+2)h_{n} + 3\gamma_{n} + 1
= \frac{3}{2}((2\gamma_{n}+1)h_{n} + 2\gamma_{n}) + \frac{1}{2}h_{n} + 1 \\
&\leq \frac{3}{2}((2\gamma_{n}+1)h_{n} + 2\gamma_{n}) + f^{*}(\gamma_{n}) + 1 \\
&<  \frac{3}{2}((2\gamma_{n}+1)h_{n} + 2\gamma_{n} + 1) + f^{*}((2\gamma_{n}+1)h_{n} + 2\gamma_{n}+1)
\end{align*}
As $p(q) - 1.5q$ is maximized at one of these two lengths in each range $h_{n} < q \leq h_{n+1}$, for all $q > h_{2}$,
\[
p(q) < 1.5q + f^{*}(q) \leq 1.5q + f(q)
\]
It remains to show total ergodicity (as Proposition \ref{finmeas} puts it on a finite measure space).

Let $A$ be a positive measure set and $t \in \mathbb{N}$ such that $T^{t}A = A$.  For $\epsilon > 0$ and $n > t$ such that $\frac{2t}{\gamma_{n}+1} < \epsilon$, define the sets
\[
Q_{n}(\epsilon) = \{ 0 \leq j < h_{n} : \mu(I_{n,j} \cap A) > (1 - \epsilon)\mu(I_{n,j}) \}
\]
If for some fixed $\epsilon > 0$, it holds that $Q_{n,\epsilon} = \emptyset$ for infinitely many $n$ then $\mu(A) = 0$ (Lemma \ref{373}) so we can also define
$
j_{n}(\epsilon) = \min \{ j \in Q_{n}(\epsilon) \}
$ for sufficiently large $n$.

Observe that for $j \geq t$,
\[
\mu(I_{n,j-t} \cap A) = \mu(T^{-t}I_{n,j} \cap A) = \mu(I_{n,j} \cap T^{t}A) = \mu(I_{n,j} \cap A)
\]
and so if $j \in Q_{n}(\epsilon)$ with $j \geq t$ then $j-t \in Q_{n}(\epsilon)$.  Therefore $j_{n}(\epsilon) < t$.  Now observe that, for $j > 0$,
\[
\mu(T^{h_{n}}I_{n,j} \cap I_{n,j-1}) \geq \sum_{i<r_{n} : s_{n,i}=1} \mu(I_{n,j-1}^{[i]}) = \frac{|\{ 0 \leq i < r_{n} : s_{n,i} = 1 \}|}{r_{n}+1}\mu(I_{n,j}) =
\Big{(}1 - \frac{1}{\gamma_{n} + 1}\Big{)}\mu(I_{n,j})
\]
since $s_{n,i} =1$ for $L_{n}\gamma_{n}$ values of $i$ and $r_{n}+1 = L_{n}(\gamma_{n}+1)$.
Then for $1 \leq s < t$ and $j \geq s$,
\begin{align*}
\mu(T^{sh_{n}}I_{n,j} \symdiff I_{n,j-s}) &\leq \sum_{u=0}^{s-1} \mu(T^{(s-u)h_{n}}I_{n,j-u} \symdiff T^{(s-u-1)h_{n}}I_{n,j-u-1}) \\
&= \sum_{u=0}^{s-1} \mu(T^{h_{n}}I_{n,j-u} \symdiff I_{n,j-u-1})
< \frac{2s}{\gamma_{n}+1} \mu(I_{n,j})
\end{align*}
and therefore
\[
\mu(T^{sh_{n}}I_{n,j} \cap I_{n,j-s}) \geq \Big{(}1 - \frac{2s}{\gamma_{n}+1}\Big{)}\mu(I_{n,j}) \geq \Big{(}1 - \frac{2t}{\gamma_{n}+1}\Big{)}\mu(I_{n,j}) > (1-\epsilon)\mu(I_{n,j})
\]

Since $L_{n-1} = m_{n-1}!$ and $L_{n-1}$ divides $h_{n} = L_{n-1}((\gamma_{n-1}+1)h_{n-1} + \gamma_{n-1})$ and $m_{n} > n > t$, we have that $t$ divides $h_{n}$ so $T^{sh_{n}}A = A$.  Then, for $1 \leq s < t$ and $0 \leq j < h_{n} - s$ with $j \in Q_{\epsilon}(n)$,
\begin{align*}
\mu(I_{n,j+s} &\cap A) = \mu(T^{sh_{n}}(I_{n,j+s} \cap A)) = \mu(T^{sh_{n}}I_{n,j+s} \cap A) \\
&\geq \mu(T^{sh_{n}}I_{n,j+s} \cap I_{n,j} \cap A)
\geq \mu(T^{sh_{n}}I_{n,j+s} \cap I_{n,j}) - \mu(I_{n,j} \setminus A) 
> (1 - \epsilon)\mu(I_{n,j}) - \epsilon \mu(I_{n,j})
\end{align*}
meaning that if $j \in Q_{n}(\epsilon)$ with $j < h_{n} - s$ then $j+s \in Q_{n}(2\epsilon)$.

Since $j \in Q_{n}(\epsilon)$ implies $j \in Q_{n}(2\epsilon)$, this means that $j_{n}(\epsilon) + kt + s \in Q_{n}(2\epsilon)$ for all $k \geq 0$ and $0 \leq s < t$ such that $j_{n}(\epsilon) + kt + s < h_{n}$.  So $Q_{n}(2\epsilon)$ contains all $j_{n}(\epsilon) \leq j < h_{n}$.
Then $|Q_{n}(2\epsilon)| \geq h_{n} - t$ so
\begin{align*}
\mu(A) &\geq \sum_{j \in Q_{n}(2\epsilon)} \mu(A \cap I_{n,j}) > |Q_{n}(2\epsilon)|(1 - 2\epsilon)\mu(I_{n,j})
\geq \Big{(}1 - \frac{t}{h_{n}}\Big{)}(1 - 2\epsilon)\mu(C_{n}) \to 1 - 2\epsilon
\end{align*}
As $\epsilon > 0$ was arbitrary, we conclude that $\mu(A) = 1$.
\end{proof}

\begin{remark}
Our proof of weak mixing does not apply when $\gamma_{n}$ is unbounded and we strongly suspect our transformations with $\gamma_{n} \to \infty$ are not weakly mixing.
\end{remark}

\section{Attaining specific complexities}\label{questions}

We conclude with a brief discussion of the main open question:

\begin{question}
For what pairs of values $1 \leq \alpha \leq \beta < 2$ does there exists a weakly mixing (rank-one or not) subshift with $\liminf \frac{p(q)}{q} = \alpha$ and $\limsup \frac{p(q)}{q} = \beta$?
\end{question}

Obviously the most interesting question is whether there exists a weakly mixing subshift, necessarily not rank-one, with $\beta < 1.5$.  We tentatively conjecture that our examples are the best possible:

\begin{conjecture}
Every subshift admitting a weakly mixing (probability) measure has complexity such that $\limsup \nicefrac{p(q)}{q} > 1.5$.
\end{conjecture}

Heinis \cite{heinis} showed that $\beta \geq 3 - \frac{2}{\alpha}$ for every subshift with $\limsup \frac{p(q)}{q} < 2$.
Our work shows that $\beta \geq 1.5$ is necessary for total ergodicity in the rank-one setting.

The values $\alpha = 1$ and $\beta = \frac{1}{4\gamma+2}$ for $\gamma \in \mathbb{N}$, $\gamma \geq 2$, are attained by our examples as they have complexity satisfying $\liminf \frac{p(q)}{q} = 1$ provided $L_{n} \to \infty$, and $\limsup \frac{p(q)}{q} = 1.5 + \frac{1}{4\gamma+2}$.

Ferenczi \cite{ferenczichacon} showed that the weakly mixing rank-one subshift given by $B_{n+1} = B_{n}^{2}1B_{n}^{2}$ has $\alpha=1.5$ and $\beta=\nicefrac{5}{3}$ (this is the example that was the previously known lowest complexity).

Our examples can be adapted to attain more pairs: for all $2 \leq m < M$, by setting $\gamma = M-1$ and $L = m$, Theorem \ref{comp} gives  a weakly mixing subshift such that
\[
\liminf \frac{p(q)}{q} = 1 + \frac{1}{M} \quad\quad\text{and}\quad\quad \limsup \frac{p(q)}{q} = \frac{3}{2} + \frac{1}{4m-2}
\]

Since $M \geq 3$ and $m \geq 2$, all of these examples satisfy $\alpha \leq \nicefrac{4}{3}$ and $\beta \leq \nicefrac{5}{3}$.

\dbibliography{ComplexityWeakMixing}

\end{document}